\documentclass{article}

\usepackage[preprint]{arxive_2026}

\usepackage[utf8]{inputenc}
\usepackage[T1]{fontenc}
\usepackage{hyperref}
\usepackage{url}
\usepackage{amsmath,amssymb,amsfonts,amsthm,mathtools}
\usepackage{nicefrac}
\usepackage{microtype}
\usepackage{xcolor}
\usepackage{graphicx}
\usepackage{algorithm}
\usepackage{algpseudocode}
\usepackage{enumitem}
\usepackage{subcaption} 
\usepackage{booktabs}   
\usepackage{capt-of}
\usepackage{float}

\newtheorem{theorem}{Theorem}
\newtheorem{lemma}{Lemma}
\newtheorem{proposition}{Proposition}
\newtheorem{corollary}{Corollary}
\theoremstyle{definition}

\theoremstyle{remark}

\newcommand{\spfun}{\operatorname{softplus}}

\newcommand{\NE}{\mathrm{NE}}

\newcommand{\methodours}{NB--HalpernSGD via LM--MultiLRSGA}
\newcommand{\Lleft}{L_{\mathrm{left}}}
\newcommand{\Lcenter}{L_{\mathrm{center}}}
\newcommand{\Lright}{L_{\mathrm{right}}}
\newcommand{\Lsum}{L_{\mathrm{sum}}}
\newcommand{\thetane}{\theta^{\mathrm{NE}}}

\title{Nash-Bargaining HalpernSGD via Limited-Memory MultiLRSGA: A Two-Phase Optimizer for Multi-Objective Learning}

\author{%
  Katherine Rossella Foglia \\
  Department of Mathematics and Computer Science \\
  University of Calabria \\
  Rende (CS), 87036, Italy \\
  \texttt{katherine.foglia@unical.it} \\
  \And
  Francesco Sergio Pisani \\
  Institute for High Performance Computing and Networking \\
  Italian National Research Council \\
  Rende (CS), 87036, Italy \\
  \texttt{francescosergio.pisani@icar.cnr.it} \\
  \And
  Vittorio Colao \\
  Department of Mathematics and Computer Science \\
  University of Calabria \\
  Rende (CS), 87036, Italy \\
  \texttt{vittorio.colao@unical.it}%
}

\begin{document}

\maketitle

\begin{abstract}
We propose NB-HalpernSGD via LM-MultiLRSGA, a two-phase optimizer for multi-objective optimization. The process starts with a competitive optimization phase by applying a limited-memory variant of MultiLRSGA, named LM-MultiLRSGA, designed to reduce the memory footprint of the original optimizer while preserving its rotational correction mechanism. It addresses multi-objective tasks by considering the competitive game associated with the original multi-objective problem, approximating a Nash equilibrium. This point defines a Nash-equilibrium-induced disagreement point: we formulate the Nash bargaining problem associated with the original losses and rewrite its Nash product as a logarithmic minimization surrogate. This minimization problem is solved using HalpernSGD, anchored at the computed competitive reference point. Therefore, the method uses the Nash equilibrium as a principled reference point for the bargaining stage and then moves toward a Pareto-oriented solution of the original multi-objective problem.
We validate the proposed optimizer on a PINN-inspired neural model, where it outperforms established multi-objective optimizers, including PCGrad, MultiAdam, and DualConeGD. Finally, while the convergence properties of HalpernSGD have been extensively studied, we discuss the convergence and stability properties of the proposed LM-MultiLRSGA phase.

\end{abstract}

\section{Introduction}
\label{sec:introduction}

Many modern multi-task learning models are not governed by a single scalar loss.
Instead, they require the simultaneous optimization of several objectives,

\begin{equation}
    \label{multiobproblem}
    \min_{\theta\in\mathbb{R}^d}
    \big(f_1(\theta),\ldots,f_h(\theta)\big),
    \qquad h\ge 2 .
\end{equation}

Since the objectives may be mutually conflicting, the natural solutions are Pareto-optimal points: \(\theta^\star\) is Pareto optimal if there is no \(\theta\) such that
\(f_i(\theta)\le f_i(\theta^\star)\) for all \(i\) and
\(f_j(\theta)<f_j(\theta^\star)\) for at least one \(j\). Existing first-order approaches usually act directly on the multi-objective
gradient geometry. DCGD \citep{hwang2024dualcone} selects an update direction inside a dual cone so that the
step is compatible with the descent of the relevant losses; MultiAdam \citep{yao2023multiadam} extends Adam by
using objective-wise adaptive moment information to balance losses with different scales; and PCGrad~\citep{yu2020pcgrad} removes destructive gradient interference by projecting conflicting task gradients.

We propose a different approach: a competitive-to-multiobjective optimizer.
Starting from a block-structured
multi-objective problem, we first consider the associated \(h\)-player differentiable game
obtained by partitioning the parameter vector as
\(\theta=(\theta_1,\ldots,\theta_h)\) and assigning player \(i\) the loss function \(f_i(\theta)\), namely
\begin{equation} \label{competproblem}
    \min_{\theta_i} f_i(\theta_1,\ldots,\theta_h)
    \qquad h\ge 2,\quad i \in \{1, \dots, h \} .
\end{equation}
This competitive phase computes an approximate Nash equilibrium \(\theta^{\rm NE}\).
Crucially, this equilibrium is not used as the final multi-objective solution. A Nash equilibrium is a unilateral notion: no player can lower its own loss by changing only its
own block while the other blocks are fixed. It need not be Pareto efficient. We therefore use
\(\theta^{\rm NE}\) only as a principled disagreement point for a second, cooperative
phase based on Nash bargaining.

\paragraph{Main Contributions.}
LM--MultiLRSGA is introduced as a limited-memory version of the recently  MultiLRSGA optimizer \citep{foglia2026multilrsga}. MultiLRSGA extends the low-rank symplectic correction of LRSGA \citep{vater2025lrsga} to \(h\ge2\) players by approximating the mixed second-order interactions through secant matrices, thereby avoiding explicit mixed Hessian computations. However, the full method still stores dense playerwise secant matrices, whose memory cost is prohibitive in large neural models. Following the limited-memory philosophy of LBFGS \citep{lbfgs} in the single objective optimization and LM--LRSGA \citep{foglia2026lmlrsga} in the two-player setting, we store only a limited history of displacement-gradient pairs and compute the required skew-correction products through rectangular two-loop recursions.\\
Our second contribution is a Nash-bargaining refinement tailored to neural training.
After the competitive phase, we set
\(
    \delta_i := f_i(\theta^{\rm NE})
\) for $i=1,\ldots,h,$
and reformulate the Nash product through a tractable minimization surrogate. Under the
classical bargaining assumptions, the Nash bargaining solution is Pareto efficient; in our nonconvex learning setting, the surrogate provides a practical Pareto-oriented cooperative refinement of the original objectives \citep{nash1950bargaining,nash1953two,cala2021nash}. We solve this surrogate with HalpernSGD anchored at \(\theta^{\rm NE}\). This choice is motivated, in convex smooth settings, by the theoretical convergence properties of Halpern-type stochastic gradient
schemes and by empirical evidence that HalpernSGD can reduce training epochs and
carbon emissions compared with SGD while preserving predictive performance
\citep{foglia2024halpernsgd,colao2026optimizer, colao2026halpernsgd}.

The rest of the paper is organized as follows. Section~\ref{sec:method}
introduces the proposed two-phase optimizer. It first presents LM--MultiLRSGA,
the limited-memory extension of MultiLRSGA used in the competitive phase to compute an approximate stable Nash equilibrium, and then describes the
Nash-bargaining surrogate optimized by HalpernSGD. Section~\ref{sec:pinn_experiments} evaluates the method on a hard-constrained domain-decomposed PINN benchmark and compares it with PCGrad, MultiAdam, and DualConeGD. Section~\ref{sec:experiments} introduces the performance discussion on a different domain presenting a GAN experiment showing that the Nash-bargaining refinement improves the quality
of the generated samples. Section~\ref{sec:conclusion} concludes the paper. The
derivation and local convergence theory of LM--MultiLRSGA, together with
additional experimental diagnostics, are collected in the appendix.

\section{Method}
\label{sec:method}

Let \(f_i:\mathbb{R}^d\to\mathbb{R}\), \(i=1,\ldots,h\), be smooth losses, and assume that
\(\theta\) is equipped with a block partition
\(\theta=(\theta_1,\ldots,\theta_h)\), \(\theta_i\in\mathbb{R}^{d_i}\),
\(d=\sum_i d_i\). We address problem~\eqref{multiobproblem} by first using this partition
to define the auxiliary \(h\)-player game~\eqref{competproblem}, whose LM--MultiLRSGA
phase reaches a practical Nash equilibrium point \(\theta^{\mathrm{NE}}\).
Second, we use the
loss vector \((f_1(\theta^{\rm NE}),f_2(\theta^{\rm NE}), \dots , f_h(\theta^{\rm NE})\)) as an endogenous disagreement point and optimize a
Nash-bargaining surrogate anchored at \(\theta^{\rm NE}\).

Denote the game gradient and its Jacobian respectively by $F(\theta)=
\big(\nabla_{\theta_1}f_1(\theta),\ldots,\nabla_{\theta_h}f_h(\theta)\big)\in\mathbb{R}^d$ and $H(\theta)=DF(\theta)=S(\theta)+A(\theta)$, where \(A(\theta)=\tfrac12(H(\theta)-H(\theta)^\top)\) is the antisymmetric component, responsible for rotational behavior. A local SNE satisfies
\(F(\theta^{\rm NE})=0\), local nonsingularity of \(H(\theta^{\rm NE})\), and positive
definiteness of the own-player Hessian blocks.
Our local theory assumes the coercivity condition
\(S(\theta^{\rm NE})\succ0\).

\subsection{Competitive phase: Limited-memory MultiLRSGA}
\label{subsec:lmmultilrsga}

The full-memory MultiLRSGA method stores, for each player \(i\), a dense matrix. In
particular, although it avoids the explicit computation of mixed second-order derivatives by
using rank-one secant approximations, at each iteration it still requires storing the \(h\)
matrices
\[
M_i^k \approx D\bigl(\partial_{\theta_i}f_i(\theta^k)\bigr)
\in \mathbb{R}^{d_i\times d},
\qquad i=1,\dots,h,
\]
from which the blocks \([M_i^k]_j \approx \nabla_{\theta_i\theta_j}^2 f_i(\theta^k)
\in \mathbb{R}^{d_i\times d_j}\) are extracted. The componentwise MultiLRSGA update is
\begin{equation}\label{eq:MLRSGA_component}
\theta_i^{k+1}
= \theta_i^k
-\eta\,\partial_{\theta_i} f_i(\theta^k)
+ \eta\,\frac{\tau}{2}
\sum_{\substack{j=1 \\ j \neq i}}^h
\Bigl([M_{i}^{k}]_j - \bigl([M_{j}^{k}]_i\bigr)^{\top}\Bigr)
\partial_{\theta_j} f_j(\theta^k),
\qquad i=1,\dots,h .
\end{equation}
The corresponding Broyden updates are
\begin{equation}
\label{eq:broyden_multilrsga}
 M_{i}^{k+1}
= M_{i}^k
+
\frac{
\Bigl(\partial_{\theta_i} f_i(\theta^{k+1})
-\partial_{\theta_i} f_i(\theta^k)
- M_{i}^k(\theta^{k+1}-\theta^k)\Bigr)
(\theta^{k+1}-\theta^k)^{\top}}
{(\theta^{k+1}-\theta^k)^{\top}(\theta^{k+1}-\theta^k)},
\quad i=1,\dots,h .
\end{equation}

Storing the matrices \(M_i^k\) becomes prohibitive in large-scale models. We therefore propose
a limited-memory variant called LM-MultiLRSGA based on the most recent \(\ell\) displacement pairs
\[
s^t:=\theta^{t+1}-\theta^{t}\in\mathbb{R}^d,
\qquad
y_i^t:=\partial_{\theta_i}f_i(\theta^{t+1})
-\partial_{\theta_i}f_i(\theta^{t})\in\mathbb{R}^{d_i},
\]
for \(t=k-\ell,\ldots,k-1\). 

In stochastic training, we use the EMA variant (like in \citep{foglia2026lmlrsga}) replacing
\(y_i^t\)  by
\(
\widetilde y_i^t
=
\beta_{\mathrm{EMA}}\widetilde y_i^{t-1}
+
(1-\beta_{\mathrm{EMA}})y_i^t,
\)
which damps mini-batch noise in consecutive gradient differences. 

Defining \(p_t:=1/((s^t)^\top s^t)\) and
\(V_t:=I-p_t s^t(s^t)^\top\), which is symmetric, the Broyden recursion over the last \(\ell\) steps gives
\begin{equation}
\label{FormuleRicorsiveLimitate}
\begin{aligned} 
M_i^k
&=M_i^{k-\ell}(V_{k-\ell}\cdots V_{k-1})
+\sum_{t=k-\ell}^{k-1}p_t\,y_i^t(s^t)^\top(V_{t+1}\cdots V_{k-1}),\\
(M_i^{k})^{\top}
&=(V_{k-1}\cdots V_{k-\ell})(M_i^{k-\ell})^{\top}
+\sum_{t=k-\ell}^{k-1}p_t\,(V_{k-1}\cdots V_{t+1})s^t(y_i^t)^{\top}.
\end{aligned} 
\end{equation}
with the convention that an empty product is the identity.

To compute the products
\(
[M_i^k]_j\,\partial_{\theta_j} f_j(\theta^k),
\) with \(i,j=1,\dots,h\), \(j\neq i\), 
we propose the rectangular two-loop recursion in
Algorithm~\ref{alg:direct-app}, following the L-BFGS idea of \citep{lbfgs} and extending
the two-player LM-LRSGA construction of \citep{foglia2026lmlrsga} to the multi-player setting.
We apply it with
\(
q^{(j)}
:=
\bigl(
0_{d_1},\dots,0_{d_{j-1}},
\partial_{\theta_j} f_j(\theta^k),
0_{d_{j+1}},\dots,0_{d_h}
\bigr)
\in\mathbb{R}^d .
\)
Indeed, \(M_i^k q^{(j)}=[M_i^k]_j\,\partial_{\theta_j} f_j(\theta^k)\). Since the exact
matrix \(M_i^{k-\ell}\) is not stored, we replace it in the recursion by the rank-one
initial matrix
\[
H_{0,i}^k
:=
\frac{
y_i^{k-\ell-1}(s^{k-\ell-1})^\top
}{
(s^{k-\ell-1})^{\top} s^{k-\ell-1}
}
\in\mathbb{R}^{d_i\times d},
\]
computed before the pair \((s^{k-\ell-1},y_i^{k-\ell-1})\) is discarded. Thus the active window has
length \(\ell\), while the effective curvature memory is \(\ell+1\) pairs.
Similarly, following the same principle, we propose the rectangular transpose counterpart in
Algorithm~\ref{alg:transpose-app} to compute the transpose products
\(
\bigl([M_j^k]_i\bigr)^\top
\partial_{\theta_j} f_j(\theta^k),
\) with \(i,j=1,\dots,h\), \(j\neq i\). In particular,
we use the identity
\(
\bigl([M_j^k]_i\bigr)^\top
\partial_{\theta_j} f_j(\theta^k)
=
\Bigl[
(M_j^k)^\top
\partial_{\theta_j} f_j(\theta^k)
\Bigr]_i 
\), and thus Algorithm~\ref{alg:transpose-app} is applied with
\(q=\partial_{\theta_j} f_j(\theta^k)\in\mathbb{R}^{d_j}\) and initial matrix \((H_{0,j}^k)^\top\). The required contribution is obtained by extracting the \(i\)-th
block of the resulting vector.
Hence all products involving \(M_i^k q\) and \((M_i^k)^\top q\) are computed without explicitly forming \(M_i^k\).

\begin{figure}[t]
\centering
\footnotesize

\begin{minipage}[t]{0.485\textwidth}
\noindent\rule{\linewidth}{0.5pt}
\captionof{algorithm}{Two-loop recursion (direct) \\ for \(r=M_i^{k}q\) with  \(q\in\mathbb{R}^{d}\)}
\label{alg:direct-app}
\vspace{-0.3em}
\noindent\rule{\linewidth}{0.5pt}
\vspace{-0.4em}
\begin{algorithmic}[1]

\State \textbf{Input:} \(\{p_t,s^t,y_i^t\}_{t=k-\ell}^{k-1}\), 
\(q\), \(H_{0,i}^{k}\)
\For{\(t = k-1, k-2, \dots, k-\ell\)}
  \State \(\alpha_t \leftarrow p_t (s^t)^{\top} q\)
  \State \(q \leftarrow q-\alpha_t s^t\)
\EndFor
\State \(r \leftarrow H_{0,i}^{k} q\)
\For{\(t = k-\ell, k-\ell+1, \dots, k-1\)}
  \State \(r \leftarrow r+y_i^t\alpha_t\)
\EndFor
\State \textbf{Output:} \(r\in\mathbb{R}^{d_i}\)
\end{algorithmic}
\vspace{-0.3em}
\noindent\rule{\linewidth}{0.5pt}
\vspace{-0.4em}
\end{minipage}
\hfill
\begin{minipage}[t]{0.485\textwidth}
\noindent\rule{\linewidth}{0.5pt}
\captionof{algorithm}{Two-loop recursion (transpose) \\ for \(r=(M_i^{k})^{\top}q\) with \(q\in\mathbb{R}^{d_i}\) }
\vspace{-0.3em}
\noindent\rule{\linewidth}{0.5pt}
\vspace{-0.4em}
\label{alg:transpose-app}
\begin{algorithmic}[1]
\State \textbf{Input:} \(\{p_t,s^t,y_i^t\}_{t=k-\ell}^{k-1}\), 
\(q\), \((H_{0,i}^{k})^{\top}\)

\For{\(t = k-1, k-2, \dots, k-\ell\)}
  \State \(\alpha_t \leftarrow p_t (y_i^t)^{\top} q\)
\EndFor
\State \(r \leftarrow (H_{0,i}^{k})^{\top}q\)
\For{\(t = k-\ell, k-\ell+1, \dots, k-1\)}
  \State \(\beta_t \leftarrow p_t (s^t)^{\top} r\)
  \State \(r \leftarrow r+s^t(\alpha_t-\beta_t)\)
\EndFor
\State \textbf{Output:} \(r\in\mathbb{R}^{d}\)
\end{algorithmic}
\vspace{-0.3em}
\noindent\rule{\linewidth}{0.5pt}
\vspace{-0.4em}
\end{minipage}
\end{figure}
Thus LM-MultiLRSGA is based on the most recent \(\ell\) displacement pairs and and one rank-one base pair and its local convergence properties are obtained by extending the two-objective LM--LRSGA argument to the \(h\)-player block setting; this extension is made explicit in Appendix~\ref{app:lm-multilrsga-theory}.

\subsection{Cooperative phase: Nash bargaining with endogenous disagreement}
\label{subsec:nb}

A Nash-bargaining problem is determined by a feasible outcome set together with a
disagreement point \citep{nash1950bargaining,nash1953two,roth1979axiomatic}. In our
setting, following the endogenous choice used in \citep{cala2021nash}, the competitive
phase first provides an approximate equilibrium \(\theta^{\NE}\). 

We then define the disagreement levels as $\delta_i:=f_i(\theta^{\NE})$, for $i=1,\dots,h$.
For a candidate parameter vector \(\theta\), define the improvement over disagreement as $g_i(\theta):=\delta_i-f_i(\theta)$.
Thus \(g_i(\theta)>0\) for all \(i\) means that \(\theta\) strictly improves all losses relative to the disagreement point, i.e., it strictly Pareto-dominates \(\theta^{\NE}\) in the objective space.

Classically, the Nash-bargaining solution is sought over the improving feasible set, and
in convex settings on the Pareto frontier, by maximizing the Nash product
\[
\max_{\theta}\;\prod_{i=1}^h g_i(\theta)
=
\max_{\theta}\;\prod_{i=1}^h\bigl(\delta_i-f_i(\theta)\bigr).
\]
On the region where all gains are positive, this is equivalent to minimizing the negative
log-Nash product $-\sum_{i=1}^h \log\bigl(\delta_i-f_i(\theta)\bigr)$.
In stochastic neural training, however, explicitly enforcing \(g_i(\theta)>0\) at every step is cumbersome. We therefore replace each raw improvement inside the logarithm by a smooth positive surrogate.
Let $\spfun_{\kappa}(x):=\frac{1}{\kappa}\log(1+e^{\kappa x})$ for $\kappa>0$ and define
\begin{equation}
\label{eq:nb-loss-main}
\mathcal L_{\mathrm{NB}}(\theta)
:= -\sum_{i=1}^h
\log\Bigl(\varepsilon + \spfun_{\kappa}\bigl(\delta_i-f_i(\theta)\bigr)\Bigr).
\end{equation}
where \(\varepsilon>0\) is a small numerical constant. 
When the gains are safely positive, \(\spfun_{\kappa}(g_i)\approx g_i\), so
\eqref{eq:nb-loss-main} behaves like the standard negative log-Nash product, while it remains well defined also outside the improving region.
In the stochastic implementation, \eqref{eq:nb-loss-main} is evaluated at the current iterate \(\theta^k\), namely with \(g_i^k:=g_i(\theta^k)=\delta_i-f_i(\theta^k)\).
We optimize \eqref{eq:nb-loss-main} with HalpernSGD
\citep{foglia2024halpernsgd,colao2026optimizer,colao2026halpernsgd}. It is an anchored stochastic
gradient method: besides the current iterate, it uses a fixed reference point \(u\), termed 'anchor', which remains fixed throughout training. For a convex \(L\)-smooth objective $\mathcal L$,
with nonempty solution set \(S=\arg\min \mathcal L\), the Halpern mechanism selects the
anchor-projected minimizer \(P_S(u)\), i.e., the minimizer closest to \(u\). Its general
stochastic iteration is
\(
\theta^{k+1}
=
\alpha_k u
+
(1-\alpha_k)
\Bigl(
\theta^k-\eta_k\,\widehat{\nabla}\mathcal L(\theta^k)
\Bigr),
\)
where \(\widehat{\nabla}\mathcal L(\theta^k)\) denotes the mini-batch gradient of the
loss \(\mathcal L\). 
Under the assumptions stated in \citep{colao2026halpernsgd}, namely the standard
Halpern conditions on \(\{\alpha_k\}\), Robbins--Monro-type conditions and coupling
requirements on \(\{\eta_k\}\), a martingale-difference stochastic gradient noise, and a
weighted square-summability condition on the conditional noise variance, the stochastic
iterates converge almost surely to \(P_S(u)\). This variance-control condition is milder
than the common uniform bounded-variance assumption and an explicit class of admissible schedules satisfying the convergence assumptions is
\[
\alpha_k=\frac{1}{k+1},
\qquad
\eta_k =\frac{\eta_0}{(k+1)^{\rho}},
\quad \text{with }
\rho\in\left(\frac12,1\right).
\]
These are the schedules used in our experiments. Beyond these convergence
properties, HalpernSGD has also been shown experimentally to reduce training time and the associated energy/carbon cost with respect to classical SGD while preserving comparable predictive performance \citep{foglia2024halpernsgd,colao2026halpernsgd}.

In our nonconvex setting, this convex theory only motivates the anchor choice. In particular we choose both the initial point and the anchor as the competitive
equilibrium, $\theta^0=u=\theta^{\NE}$, and we take \(\mathcal L=\mathcal L_{\mathrm{NB}}\). Therefore the actual phase-2 iteration is
\begin{equation}
\label{eq:halpern-main}
\theta^{k+1}
=
\alpha_k\,\theta^{\NE}
+
(1-\alpha_k)
\Bigl(
\theta^k-\eta_k\,\widehat{\nabla}\mathcal L_{\mathrm{NB}}(\theta^k)
\Bigr),
\qquad
\theta^0:=\theta^{\NE}.
\end{equation}
This choice ties the bargaining refinement to the same point that defines the
endogenous disagreement levels \(\delta_i=f_i(\theta^{\NE})\), stabilizing the transition
from the competitive phase to the Nash-bargaining phase.

\section{Experiments on hard-constrained domain-decomposed PINNs}
\label{sec:pinn_experiments}

This section evaluates the proposed two-phase optimizer on a controlled
multi-objective PINN benchmark. The goal is to isolate the optimization mechanism from the enforcement of physical constraints. Therefore, initial and boundary conditions are imposed by construction, while the active multi-objective problem is defined only by three homogeneous local PDE residual losses. We evaluated the optimizer on two problems: training a PINN-based model, which allows us to assess performance from a behavioral standpoint during the training process, and a GAN-based approach, where the qualitative outcome of the trained model is analyzed. Although limited, these experiments provide an initial assessment of the effectiveness of the proposed approach and serve as a starting point for further analyzes with more complex case studies.

The experiments were performed on a machine equipped with an NVIDIA GeForce RTX 3090 GPU with 24 GB of dedicated memory, a 12-core AMD Ryzen 9 5900X processor, 64 GB of RAM, and Fedora Linux as the operating system. The code to reproduce our results is available at this link\footnote{Link will be available for the camera-ready version.}.

\subsection{Problem setting}
\label{subsec:pinn_problem_setting}

We consider the one-dimensional viscous Burgers equation
\(u_t+u\,u_x-\nu u_{xx}=0\) on \([0,1]\times[-1,1]\), with
\(\nu=0.01/\pi\), initial condition \(u(0,x)=u_0(x)=-\sin(\pi x)\), and homogeneous boundary conditions \(u(t,\pm1)=0\). 
Following the classical trial-solution construction of ~\citep{lagaris1998artificial}
and the recent use of algebraic inclusion of boundary/initial conditions in PINNs~\citep{ren2024improving},
we impose the initial and boundary conditions by construction through
$ u_\theta(t,x)=(1-t)u_0(x)+t(1-x^2)N_\theta(t,x)$.
Indeed, this gives \(u_\theta(0,x)=u_0(x)\) and \(u_\theta(t,\pm1)=0\) identically;
therefore, no initial-condition or boundary-condition loss is optimized. The active training objectives are only the PDE residual losses. The PDE residual is defined as
\( r_\theta(t,x)=\partial_tu_\theta+u_\theta\partial_xu_\theta-\nu\partial_{xx}u_\theta\)
 and we split the spatial domain into \[\Omega_1=\{x<-\tfrac{1}{3}\},\quad
\Omega_2=\{-\tfrac{1}{3}\le x<\tfrac{1}{3}\},\quad
\Omega_3=\{x\ge\tfrac{1}{3}\}, \] and minimize the three local residual objectives
\[
    L_i(\theta)=\frac{1}{|\Omega_i|}
    \sum_{(t,x)\in\Omega_i} r_\theta(t,x)^2,
    \quad i=1,2,3.
\] 
We denote these losses by
\((\Lleft,\Lcenter,\Lright)=(L_1,L_2,L_3)\) and their sum by  
\(\Lsum\). The sum of the losses facilitates a more straightforward comparison.
Moreover, to ensure numerical stability of the log-softplus computation, we use the value of \textit{5} for the parameter \(\kappa\) in Equation \eqref{eq:nb-loss-main} across all experiments.

\subsection{Hard-constrained partition-of-unity PINN}
\label{subsec:pinn_model}


The neural correction term \(N_\theta\) is modeled as a smooth partition-of-unity mixture of three local experts,
\[
    N_\theta(t,x)=\sum_{i=1}^3 g_i(x)E_i(t,x),
    \qquad
    g_i(x)=
    \frac{\exp\{-\beta(x-c_i)^2\}}
    {\sum_{j=1}^3\exp\{-\beta(x-c_j)^2\}},
\]
with centers \((c_1,c_2,c_3)=(-2/3,0,2/3)\) and \(\beta=12\). The gates satisfy
\(g_i(x)\geq0\) and \(\sum_i g_i(x)=1\), so the three experts are smoothly blended
without introducing discontinuities at the subdomain interfaces.

Each expert \(E_i\) is a fully connected \(\tanh\) MLP with architecture
\(2\to15\to15\to15\to1\). Denoting by \(\theta_i\) the parameters of the
\(i\)-th expert, the full parameter vector is
\(\theta=(\theta_1,\theta_2,\theta_3)\). This block structure is used only by the
competitive phase of our optimizer, where each expert block is interpreted as one
player. In the cooperative Halpern phase, and for all baselines, the same architecture is optimized as a single global PINN.

\subsection{Optimizers and training protocol}
\label{subsec:pinn_training_protocol}

The proposed method is compared with PCGrad, MultiAdam, and DCGD (all variants: DCGD-center, DCGD-avg, and DCGD-proj). All methods use the same architecture, the same three local residual objectives, the same 
mini-batch protocol, and the same training epochs.

For the proposed method, the first phase treats the three expert blocks as three
players and solve the competitive problem:
\begin{equation}
    \min_{\theta_1} \Lleft(\theta_1,\theta_2,\theta_3),
    \qquad
    \min_{\theta_2} \Lcenter(\theta_1,\theta_2,\theta_3),
    \qquad
    \min_{\theta_3} \Lright(\theta_1,\theta_2,\theta_3).
    \label{eq:pinn_competitive_game}
\end{equation}
thereby obtaining an approximate Nash equilibrium denoted by \(\theta^{\NE}\).
The LM--MultiLRSGA phase is run with fixed parameters $\eta_{\mathrm{LM}}$, $\tau_{\mathrm{LM}}$ and $memory\_history$, i.e. the $\ell$ parameter in the Section \ref{subsec:lmmultilrsga}.

The switch to the second phase is triggered by a practical Nash stationarity
criterion: for a prescribed tolerance
\(\varepsilon_{\mathrm{Nash}}>0\), called the Nash target, we switch as soon as
\begin{equation}
    \max_{i=1,\ldots,3}
    \operatorname{RMS}\!\left(\nabla_{\theta_i} L_i(\theta)\right)
    =
    \max_{i=1,\ldots,3}
    \frac{\left\|\nabla_{\theta_i} L_i(\theta)\right\|_2}{\sqrt{d_i}}
    \leq
    \varepsilon_{\mathrm{Nash}},
    \label{eq:pinn_switch_criterion}
\end{equation}
where \(d_i\) is the dimension of the parameter block \(\theta_i\). This rule checks only the first-order necessary condition \(F(\theta)\approx 0\) in block RMS norm; it does not ensure second-order stability or the coercivity condition \(S^\ast\succ0\), whose
evaluation would require Hessian information. Then the obtained point $\theta^{\NE}$ is more precisely a practical Nash-stationary reference point.

The resulting point \(\thetane\)
is used, in the second phase, both as initialization and as the Halpern anchor. Moreover, the local losses at
\(\thetane\) define the disagreement point $\delta_i = L_i(\thetane)$ for $i=1,2,3.$

For the Halpern phase, we use the decreasing stepsize sequence
\(\eta_k=\eta_0/(k+1)^{0.5001}\), where \(k\) denotes the training step and
\(\eta_0\in\{10^{-2},10^{-3},10^{-4}\}\).
We remark that these initial values \(\eta_0\) are swept only for the second phase of the proposed
method, and for fairness, the same decreasing sequence and the same initial values \(\eta_0\)
are also used for the SGD-type baselines, namely PCGrad and the DualCone variants.
MultiAdam is run with the same nominal initial values, but keeps its internal adaptive dynamics. \\
The training set consists of $j^2=900$ interior points $(t,x)\in(0,1]\times(-1,1)$ with $j=30$, split into three spatial subdomains for the residual losses and sampled using stratified mini-batches of size $300$ with $100$ points per subdomain, while the test set is built analogously with $j_{\mathrm{test}}=45$, yielding $45^2=2025$ interior points.
Training curves are reported as epoch-averaged mini-batch
losses.

\subsection{Experiment 1: Main Benchmark}
\label{subsec:pinn_exp1_main_benchmark}

We compare the proposed method with PCGrad, MultiAdam, and the three DualCone
variants for the learning-rate values
\(\eta_0\in\{10^{-2},10^{-3},10^{-4}\}\)
and for the first phase of the proposed method we set $\eta_{\mathrm{LM}}=0.1$, $\tau_{\mathrm{LM}}=0.01$ and $memory\_history=3$. All methods are trained for $500$ iterations, run over $10$ different seeds and evaluated on the held-out test grid.
 For HalpernSGD via LM--MultiLRSGA, this budget is split as
\(
\text{phase-1 LM--MultiLRSGA steps}
+
\text{phase-2 HalpernSGD steps}
=
500 \).

The switch from phase 1 to phase 2 follows \eqref{eq:pinn_switch_criterion} with \(  \varepsilon_{\mathrm{Nash}}=10^{-2}\), which is reached after
\(35.5\pm18.1\) iterations, equivalently \(11.8\pm6.0\) mini-batch epochs, on
average over the ten seeds.

Table~\ref{tab:pinn_exp1_test_losses} reports test losses as
mean \(\pm\) standard deviation over ten seeds. The column $\Delta$ contains the percentage differences relative to the best value achieved for $\Lsum$ for each learning rate.

\begin{table}
\centering
\caption{
Test losses for the hard-constrained Burgers PINN benchmark.
All values are reported as mean \(\pm\) standard deviation over ten seeds.
Lower values are better. The column $\Delta$ contains the percentage difference relative to the best value achieved for $\Lsum$.
}
\label{tab:pinn_exp1_test_losses}
\resizebox{\textwidth}{!}{%
\begin{tabular}{llccccc}
\toprule
LR & Method  & \(\Lleft\) & \(\Lcenter\) & \(\Lright\) & \(\Lsum\) & $\Delta$ (best) \\
\midrule
\(10^{-2}\) & \methodours & \(0.1859\pm0.0705\) & \(0.8123\pm0.0292\) & \(0.1551\pm0.0633\) & \(1.1533\pm0.0797\) & \(0.0\%\) \\
\(10^{-2}\) & PCGrad & \(0.2455\pm0.0855\) & \(0.8515\pm0.0382\) & \(0.1180\pm0.0158\) & \(1.2151\pm0.0846\) & \(-5.4\%\) \\
\(10^{-2}\) & MultiAdam & \(0.4082\pm0.2926\) & \(0.8028\pm0.1645\) & \(0.5137\pm0.2062\) & \(1.7247\pm0.3459\) & \(-49.5\%\) \\
\(10^{-2}\) & DualCone center & \(0.2012\pm0.0829\) & \(0.8511\pm0.0368\) & \(0.2122\pm0.0772\) & \(1.2646\pm0.0974\) & \(-9.6\%\) \\
\(10^{-2}\) & DualCone avg & \(0.3151\pm0.0564\) & \(0.9517\pm0.0346\) & \(0.3161\pm0.0945\) & \(1.5829\pm0.1589\) & \(-37.2\%\) \\
\(10^{-2}\) & DualCone proj & \(0.3129\pm0.0575\) & \(0.9444\pm0.0413\) & \(0.3147\pm0.0927\) & \(1.5720\pm0.1621\) & \(-36.3\%\) \\
\midrule
\(10^{-3}\) & \methodours & \(0.2509\pm0.1056\) & \(0.8529\pm0.0350\) & \(0.1889\pm0.1144\) & \(1.2927\pm0.0725\) & \(0.0\%\) \\
\(10^{-3}\) & PCGrad & \(0.5248\pm0.0993\) & \(0.9604\pm0.0471\) & \(0.3878\pm0.1307\) & \(1.8730\pm0.1672\) & \(-44.9\%\) \\
\(10^{-3}\) & MultiAdam & \(0.5160\pm0.3818\) & \(0.7056\pm0.1096\) & \(0.7054\pm0.5435\) & \(1.9270\pm0.7098\) & \(-49.1\%\) \\
\(10^{-3}\) & DualCone center & \(0.5651\pm0.1765\) & \(0.9303\pm0.0506\) & \(0.4774\pm0.1559\) & \(1.9727\pm0.1857\) & \(-52.6\%\) \\
\(10^{-3}\) & DualCone avg & \(0.6602\pm0.2543\) & \(0.9649\pm0.0558\) & \(0.5529\pm0.1964\) & \(2.1779\pm0.2774\) & \(-68.5\%\) \\
\(10^{-3}\) & DualCone proj & \(0.6625\pm0.2535\) & \(0.9622\pm0.0585\) & \(0.5558\pm0.1941\) & \(2.1805\pm0.2758\) & \(-68.7\%\) \\
\midrule
\(10^{-4}\) & \methodours & \(0.2694\pm0.1145\) & \(0.8657\pm0.0392\) & \(0.1985\pm0.1247\) & \(1.3336\pm0.0625\) & \(0.0\%\) \\
\(10^{-4}\) & PCGrad & \(0.7573\pm0.3694\) & \(0.9652\pm0.0615\) & \(0.6161\pm0.2273\) & \(2.3385\pm0.3997\) & \(-75.3\%\) \\
\(10^{-4}\) & MultiAdam & \(0.7607\pm0.2909\) & \(0.6634\pm0.0206\) & \(0.6620\pm0.2603\) & \(2.0861\pm0.2612\) & \(-56.4\%\) \\
\(10^{-4}\) & DualCone center & \(0.7700\pm0.3910\) & \(0.9621\pm0.0619\) & \(0.6312\pm0.2307\) & \(2.3633\pm0.4113\) & \(-77.2\%\) \\
\(10^{-4}\) & DualCone avg & \(0.7878\pm0.4165\) & \(0.9678\pm0.0644\) & \(0.6452\pm0.2392\) & \(2.4008\pm0.4513\) & \(-80.0\%\) \\
\(10^{-4}\) & DualCone proj & \(0.7881\pm0.4164\) & \(0.9675\pm0.0647\) & \(0.6456\pm0.2389\) & \(2.4012\pm0.4511\) & \(-80.1\%\) \\
\bottomrule
\end{tabular}
}
\end{table}

Figure~\ref{fig:pinn_exp1_lsum_all_lr} shows the training evolution of the aggregate
residual loss \(\Lsum\). Curves are computed from the mini-batch losses used by
the optimizers and then averaged per epoch and over the ten seeds. The
corresponding local residual curves are reported in
Appendix~\ref{app:pinn_exp1_local_curves}.

\begin{figure*}[t!]
\centering

\begin{minipage}{0.8\textwidth}
\centering
\includegraphics[width=\linewidth]{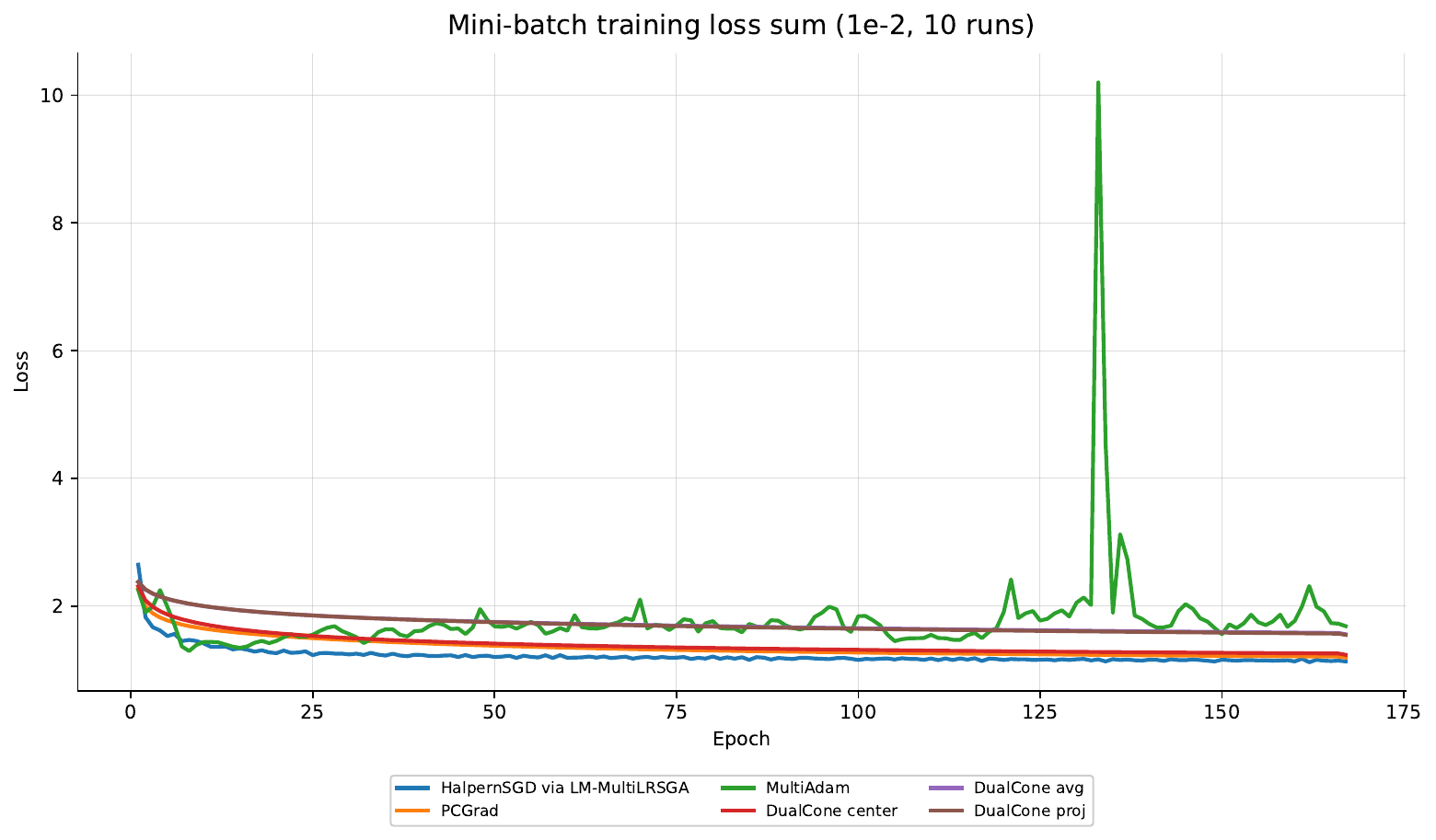}
\centerline{\small (a) \(\eta_0=10^{-2}\)}
\end{minipage}

\vspace{0.5em}

\begin{minipage}{0.49\textwidth}
\centering
\includegraphics[width=\linewidth]{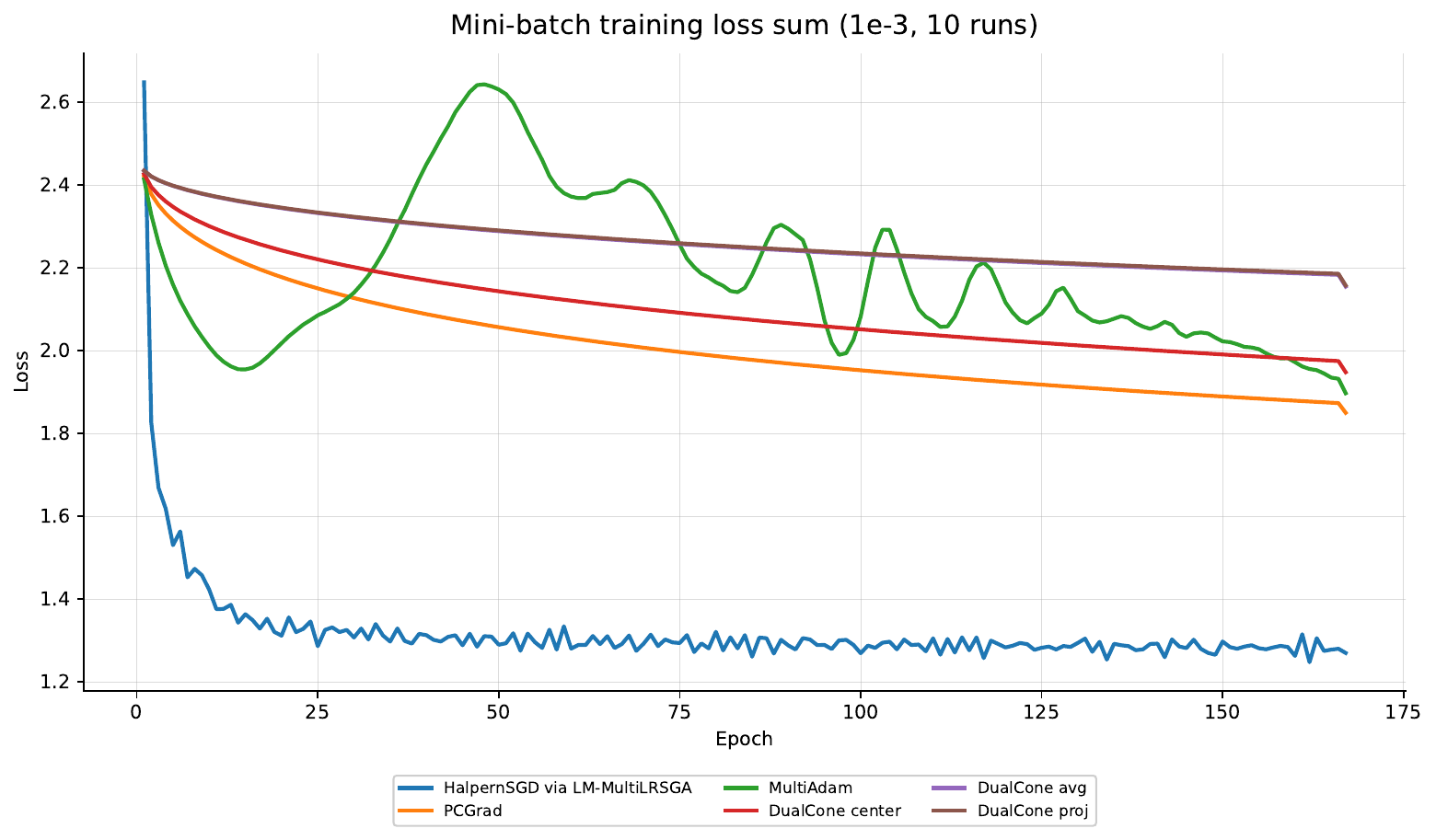}
\centerline{\small (b) \(\eta_0=10^{-3}\)}
\end{minipage}
\hfill
\begin{minipage}{0.49\textwidth}
\centering
\includegraphics[width=\linewidth]{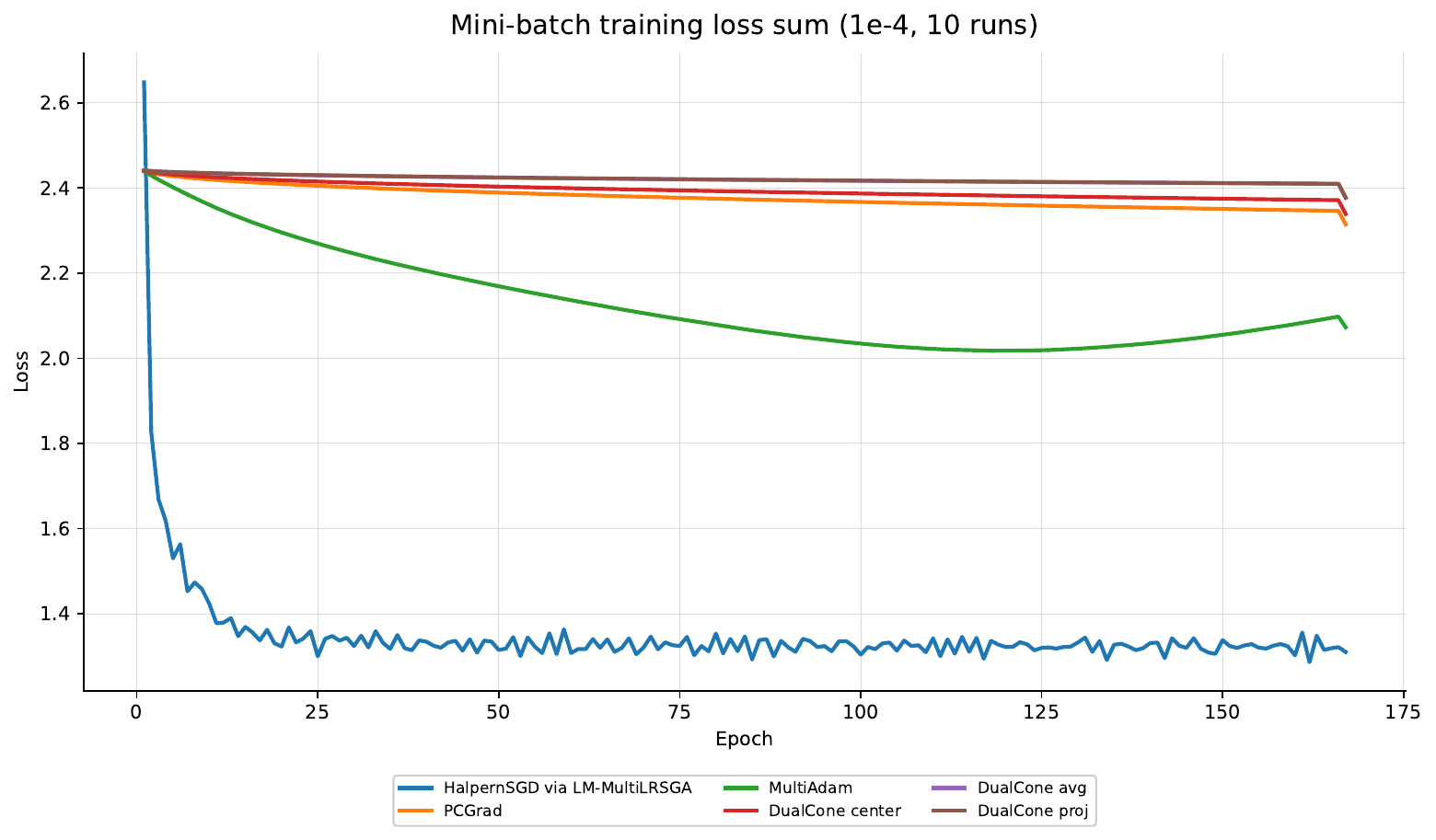}
\centerline{\small (c) \(\eta_0=10^{-4}\)}
\end{minipage}

\caption{
Experiment~1. Mean epoch-averaged mini-batch training curves for the aggregate
residual loss \(\Lsum\) over ten seeds. The proposed NB--HalpernSGD via
LM--MultiLRSGA is compared with PCGrad, MultiAdam, and DualConeGD variants
for the three initial phase-2 learning rates.
}
\label{fig:pinn_exp1_lsum_all_lr}
\end{figure*}


\subsection{Experiment 2: sensitivity to the Nash switch threshold}
\label{subsec:pinn_exp2_switch_sensitivity}

The second experiment studies the robustness of the proposed two-phase procedure with
respect to the switching threshold in \eqref{eq:pinn_switch_criterion}. The experiment uses a single seed,
so that the switch markers remain visually aligned with
the corresponding LM--MultiLRSGA trajectory.
The tested Nash stationarity thresholds are $7.5\cdot10^{-2}$, $5\cdot10^{-2}$, $10^{-2}$, $7.5\cdot10^{-3}$ and $5\cdot10^{-3}$.
Before fixing the phase-1 parameters for this diagnostic, we inspected the LM--MultiLRSGA-only dynamics for
\(\eta_{\mathrm{LM}}\in\{0.1,0.05,0.01\}\) with
\(\tau_{\mathrm{LM}}=\eta_{\mathrm{LM}}/10\). The three choices exhibit comparable
qualitative trends, although the largest step size naturally produces more oscillatory
trajectories. This is not problematic in the main benchmark, where the role of phase 1
is only to reach the practical Nash stationarity condition as quickly as possible. For the
switch-sensitivity plot, however, we use the less oscillatory setting
\(\eta_{\mathrm{LM}}=10^{-2}\) and \(\tau_{\mathrm{LM}}=10^{-3}\), so that the
effect of changing the Nash threshold is visually clearer. The corresponding
LM--MultiLRSGA-only step-size diagnostic is reported in
Appendix~\ref{app:pinn_lm_eta_diagnostic}.

In this diagnostic experiment, the Halpern phase is initialized with a learning rate of \(10^{-2}\). Figure \ref{fig:pinn_exp2_switch_lsum} also reports the corresponding LM--MultiLRSGA-only trajectory as a reference baseline. The curves are obtained by applying a moving average with window size 20 to the mini-batch losses, while circle markers denote the switching points between the phases. The plot further indicates that the method is robust with respect to the choice of Nash target, since all trajectories converge to quite similar values.

The purpose of this experiment is not to provide a second benchmark against baselines,
but to verify that the final Halpern refinement is not overly sensitive to the exact
moment at which the approximate Nash point is used as disagreement point and anchor.
The local residual components for this diagnostic are reported in
Appendix~\ref{app:pinn_exp2_local_curves}.

\begin{figure}
\centering
\includegraphics[width=0.70\linewidth]{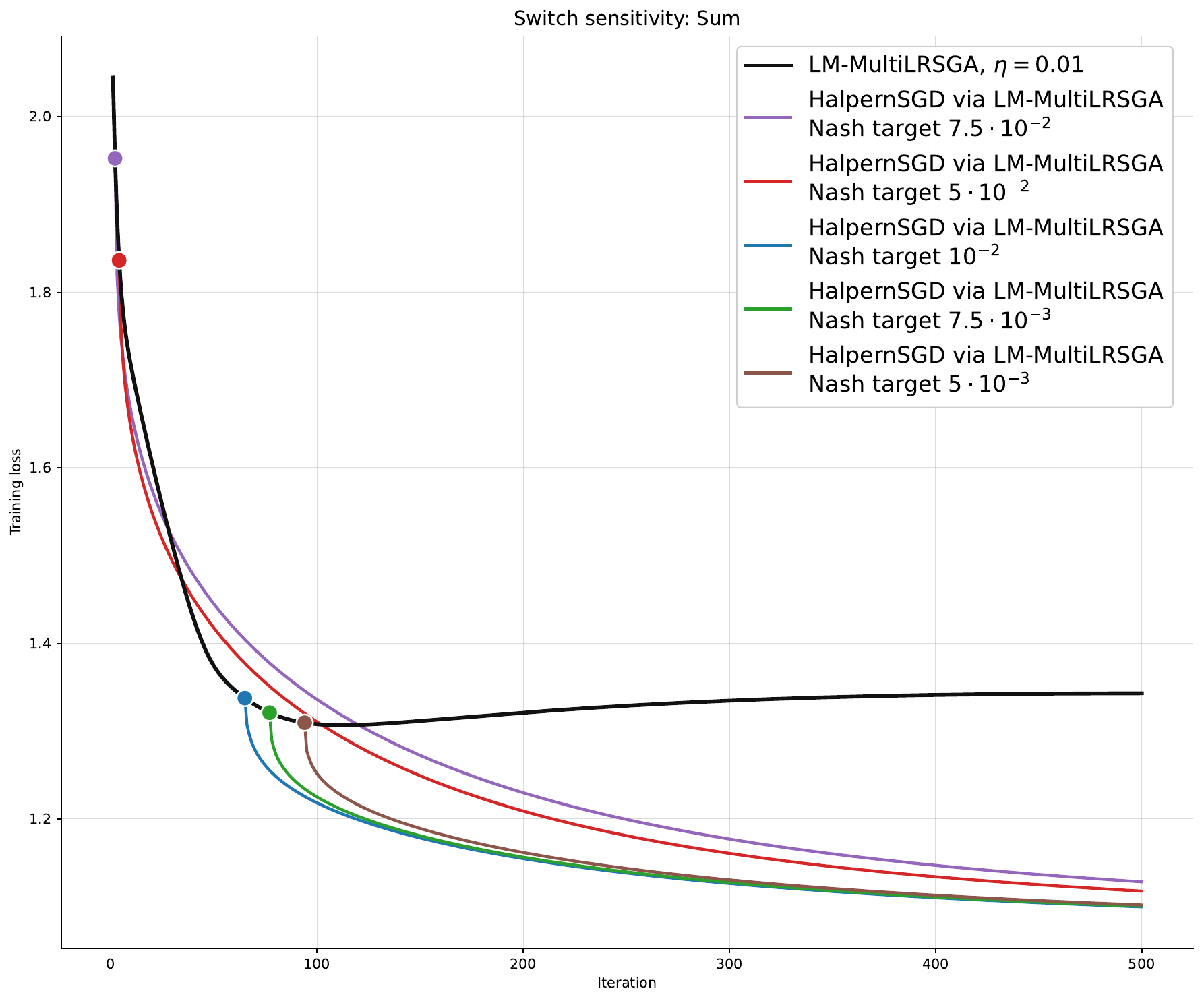}
\caption{
Experiment~2. Switch-sensitivity diagnostic for the aggregate residual loss
\(\Lsum\). Each colored curve corresponds to a different Nash stationarity target.
Markers indicate the switch from LM--MultiLRSGA to the Halpern phase.
}
\label{fig:pinn_exp2_switch_lsum}
\end{figure}

\section{Experiment on GANs}
\label{sec:experiments}

We test the effectiveness of the method on the two-objective optimization problem of training a GAN model.
In this setting, the learning dynamics can be naturally interpreted through a multi-objective optimization lens, where the generator and discriminator define two competing objectives that must be optimized simultaneously. The generator seeks to minimize a divergence between the generated and real data distributions, while the discriminator aims to maximize its ability to distinguish real from synthetic samples. This adversarial interplay induces a non-cooperative, non-convex optimization problem in which improving one objective may degrade the other, leading to well-known challenges such as cycling, instability, and sensitivity to hyperparameters. From a multi-objective perspective, the goal is not to optimize a single scalar loss, but rather to reach a balanced solution corresponding to a stable equilibrium where neither player can locally improve its own objective by a unilateral deviation while the other player's parameters are held fixed. 
For this experiment, we train a GAN model on the MNIST dataset.
 The architecture, dataset, and phase-1 adversarial setup follow the LM--LRSGA GAN experiment of~\citep{foglia2026lmlrsga}; here we use this setting only to evaluate the additional NB--HalpernSGD refinement.

We evaluated our optimizer using the Fréchet Inception Distance (FID) score \citep{heusel2017gans} and the training loss curves.
The results in Figure \ref{fig:gan_exp_loss} highlight consistent improvements in GAN training dynamics when incorporating HalpernSGD within the LM-MultiLRSGA framework. As shown in the FID curves, both methods exhibit a rapid decrease in FID during early epochs, indicating efficient initial learning due to LM-MultiLRSGA algorithm; however, the HalpernSGD refinement, activated at epoch 40, achieves a consistently lower FID and reaches the best final score.
This suggests improved sample quality and more effective exploration of the generator’s parameter space. In addition to this, the loss trajectories reveal stable and well-balanced adversarial training for both approaches, with generator losses gradually decreasing and discriminator losses approaching a steady plateau. Notably, the HalpernSGD variant maintains slightly lower generator loss and smoother convergence, indicating reduced oscillatory behavior, a common issue in GAN optimization. Together, these observations suggest that the proposed method improves both convergence stability and generative performance, supporting the role of the second phase as a refinement over the phase-1-only baseline in the
two-player setting.

\begin{figure}[t]
    \centering
    \begin{subfigure}[t]{0.49\linewidth}
        \centering
        \includegraphics[width=\linewidth]{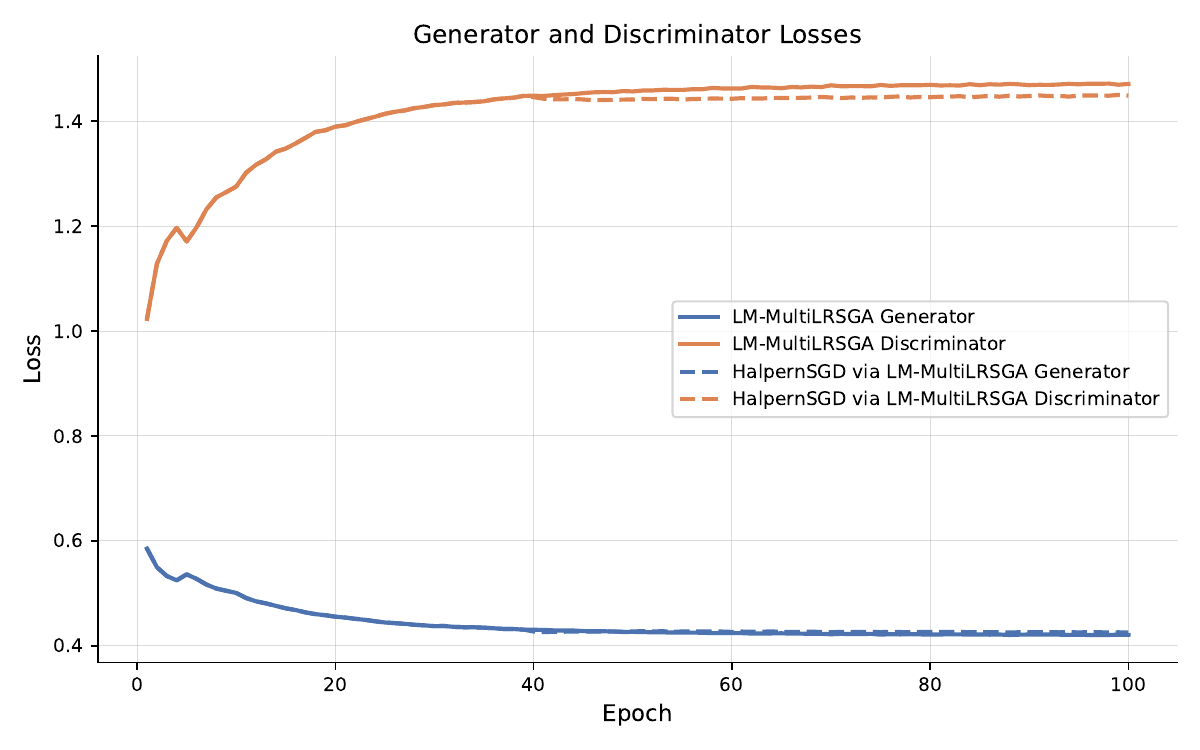}
        \centerline{\small (a)}
    \end{subfigure}
    \hfill
    \begin{subfigure}[t]{0.49\linewidth}
        \centering
        \includegraphics[width=\linewidth]{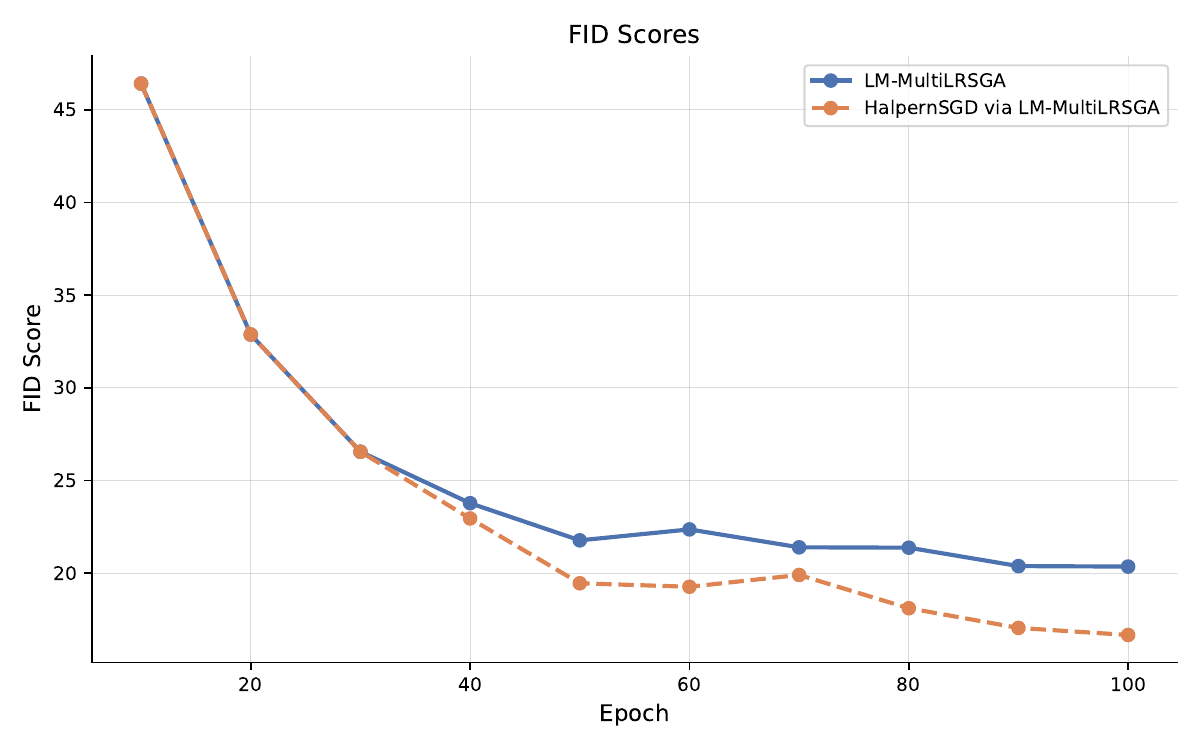}

        \centerline{\small (b)}
    \end{subfigure}

    \caption{Generator and discriminator loss trajectories and FID scores over training epochs on MNIST. Both methods exhibit stable adversarial dynamics, while NB--HalpernSGD via LM--MultiLRSGA achieves smoother convergence and consistently lower FID values, indicating faster convergence and improved sample quality.}
    \label{fig:gan_exp_loss}
\end{figure}

\section{Conclusion}
\label{sec:conclusion}

We introduced NB--HalpernSGD via LM--MultiLRSGA, a two-phase optimizer for block structure multi-objective learning. The first phase aims to compute an approximate stable Nash equilibrium through the proposed limited-memory variant of MultiLRSGA, which reduces the memory footprint of the MultiLRSGA skew-correction mechanism while retaining its game-aware structure. The second phase uses this equilibrium as both an endogenous disagreement point and the HalpernSGD anchor, yielding a Nash-bargaining refinement of the original objectives.
The hard-constrained PINN benchmark shows that the proposed method improves over PCGrad, MultiAdam, and DualConeGD across all tested learning-rate regimes, while the GAN experiment provides further evidence of smoother two-objective dynamics and improved FID behavior.
The local LM--MultiLRSGA convergence analysis, detailed in the appendix, shows how the two-objective LM--LRSGA argument extends to the \(h\)-player setting through the MultiLRSGA block estimate. 
Future work will focus on sharpening the LM--MultiLRSGA convergence theory,
including a finer characterization of the role of the memory length \(\ell\) and of the admissible stepsize region. On the experimental side, we plan to extend
the benchmark campaign to larger-scale multi-task learning models.

\bibliographystyle{plainnat}
\bibliography{references}

\clearpage
\newpage
\appendix
\section{Limited-memory MultiLRSGA: derivation and local theory}
\label{app:lm-multilrsga-theory}

This appendix makes explicit the local convergence argument for the
LM--MultiLRSGA phase introduced in Section~\ref{sec:method}. 
The analysis is deterministic and is stated for exact secant pairs. The EMA variant used in
stochastic training is a practical variance-reduction modification and is not claimed to be covered
by the deterministic convergence theorem below.
The proof
strategy follows the two-objective LM--LRSGA analysis of
\citep{foglia2026lmlrsga}. The main new technical point is the passage
from two objectives to \(h\) players, which is handled through the block estimate
of full-memory MultiLRSGA in \citep{foglia2026multilrsga}.

Throughout this appendix, we adopt the following notation:
\[
F(\theta)=
\bigl(
\partial_{\theta_1}f_1(\theta),\ldots,
\partial_{\theta_h}f_h(\theta)
\bigr)\in\mathbb R^d,
\quad
H(\theta)=DF(\theta).
\]
At a stable Nash equilibrium \(\theta^{\NE}\), set
\[
H^*:=H(\theta^{\NE})=S^*+A^*,
\qquad
S^*:=\frac12(H^*+H^{*\top}),
\qquad
A^*:=\frac12(H^*-H^{*\top}).
\]
Thus \(A^*=-A^{*\top}\). For the spectral part of the argument we assume
\(S^*=S^{*\top}\succ0\), as in the local spectral condition used in
\citep{foglia2026lmlrsga}.

For a fixed history length \(\ell\), denote by \(M_{i,k}^{(\ell)}\) the
limited-memory approximation of \(M_i^k\). We recall that, for every $i \in \{1, \dots h\}$, the matrices \(M_{i,k}^{(\ell)}\) are obtained from the last \(\ell\)
pairs $s^t:=\theta^{t+1}-\theta^t\in\mathbb R^d$ and $y_i^t:=
\partial_{\theta_i}f_i(\theta^{t+1}) -\partial_{\theta_i}f_i(\theta^t)\in\mathbb R^{d_i}$ with \(p_t:=((s^t)^\top s^t)^{-1}\) and
\(V_t:=I-p_t s^t(s^t)^\top\).

Replacing the unavailable matrix
\(M_i^{k-\ell}\) in \eqref{FormuleRicorsiveLimitate} by \(H_{0,i}^k\), we obtain
\[
\begin{aligned}
M_{i,k}^{(\ell)}
&=
H_{0,i}^k(V_{k-\ell}\cdots V_{k-1})
+
\sum_{t=k-\ell}^{k-1}
p_t\,y_i^t(s^t)^\top(V_{t+1}\cdots V_{k-1}),\\
(M_{i,k}^{(\ell)})^\top
&=
(V_{k-1}\cdots V_{k-\ell})(H_{0,i}^k)^\top
+
\sum_{t=k-\ell}^{k-1}
p_t\,(V_{k-1}\cdots V_{t+1})s^t(y_i^t)^\top .
\end{aligned}
\]
Recall that here an empty product is the identity and
\(
H_{0,i}^k
:=
\frac{y_i^{k-\ell-1}(s^{k-\ell-1})^\top}
{(s^{k-\ell-1})^\top s^{k-\ell-1}} .
\)

The limited-memory skew correction is the block matrix
\(\overline A_k^{(\ell)}\in\mathbb R^{d\times d}\) defined by
\[
[\overline A_k^{(\ell)}]_{ii}=0,
\qquad
[\overline A_k^{(\ell)}]_{ij}
=
\frac12
\Bigl(
[M_{i,k}^{(\ell)}]_j
-
([M_{j,k}^{(\ell)}]_i)^\top
\Bigr),
\qquad i\neq j .
\]
It is block-skew-symmetric by construction. The LM--MultiLRSGA iteration can
therefore be written as
\[
\theta^{k+1}
=
\theta^k-\eta\bigl(I-\tau\overline A_k^{(\ell)}\bigr)F(\theta^k).
\]

\paragraph{Imported spectral facts.}
We first recall two results from the LM--LRSGA analysis
\citep{foglia2026lmlrsga}. They are purely spectral and dimension-independent,
so they apply to the present \(d\)-dimensional \(h\)-player game.

\begin{lemma}[Symmetric-part identity, Lemma 3.1 of \citep{foglia2026lmlrsga}]
\label{lem:app-sym-identity}
Let \(H^*=S^*+A^*\), with \(S^*=S^{*\top}\succ0\) and
\(A^*=-A^{*\top}\). For every skew-symmetric matrix \(K=-K^\top\) and every
\(\tau>0\),
\[
\operatorname{sym}\bigl((I-\tau K)H^*\bigr)
=
S^*
+\frac{\tau}{2}(S^*K-KS^*)
-\frac{\tau}{2}(KA^*+A^*K).
\]
\end{lemma}

\begin{proposition}[Frozen spectral stability, Proposition 3.2 of \citep{foglia2026lmlrsga}]
\label{prop:app-frozen-stability}
Let \(\mathcal A_k=-\mathcal A_k^\top\) be a skew correction and set
\[
\Delta_k:=\mathcal A_k-A^*,
\qquad
G_k:=(I-\tau\mathcal A_k)H^* .
\]
Consider the \(k\)-frozen map
\(
T_k(\theta):=\theta-\eta(I-\tau\mathcal A_k)F(\theta)\), whose Jacobian is
$DT_k(\theta^{\NE})=I-\eta G_k$.
If
\[
0<\tau<
\frac{2\lambda_{\min}(S^*)}
{
\|S^*A^*-A^*S^*\|_2
+
\|S^*\Delta_k-\Delta_kS^*\|_2
+
\|\Delta_kA^*+A^*\Delta_k\|_2
},
\]
then there exists \(\bar\eta_k>0\) such that, for every
\(0<\eta<\bar\eta_k\),
it holds $\rho\!\left(DT_k(\theta^{\NE})\right)<1 $.
\end{proposition}

\paragraph{The multi-player block estimate.}
The next result is the key estimate from full-memory MultiLRSGA
\citep{foglia2026multilrsga}. It is the point where the already mentioned factor \((h-1)\) enters the analysis.

\begin{lemma}[Multi-player block estimate, Lemma 1 of \citep{foglia2026multilrsga}]
\label{lem:app-multilrsga-block}
Assume that, for some \(\delta>0\),
\[
\|M_i^k-D(\partial_{\theta_i}f_i)(\theta^{\NE})\|_2\le\delta,
\qquad i=1,\dots,h .
\]
Then the full-memory MultiLRSGA antisymmetric correction \(\widehat A_k\)
satisfies
\[
\|\widehat A_k-A^*\|_2\le(h-1)\delta .
\]
\end{lemma}

The same block argument applies verbatim to the limited-memory matrices
\(M_{i,k}^{(\ell)}\), because the proof of
Lemma~\ref{lem:app-multilrsga-block} uses only the block structure of the
matrices entering the skew correction. Thus, if
\[
\|M_{i,k}^{(\ell)}-D(\partial_{\theta_i}f_i)(\theta^{\NE})\|_2\le\delta,
\qquad i=1,\dots,h,
\]
then
\[
\|\overline A_k^{(\ell)}-A^*\|_2\le(h-1)\delta .
\]

Combining Proposition~\ref{prop:app-frozen-stability} with Lemma~\ref{lem:app-multilrsga-block} gives the \(h\)-player version of
Corollary 3.3 in \citep{foglia2026lmlrsga}.


\begin{corollary}[Global \(\tau\)-bound under playerwise secant accuracy]
\label{cor:app-playerwise-tau}
Assume that, for some \(\delta_{\mathrm{pl}}>0\),
\[
\|M_{i,k}^{(\ell)}-D(\partial_{\theta_i}f_i)(\theta^{\NE})\|_2
\le\delta_{\mathrm{pl}},
\qquad i=1,\dots,h .
\]
If
\[
0<\tau<
\frac{2\lambda_{\min}(S^*)}
{
\|S^*A^*-A^*S^*\|_2
+
2(\|S^*\|_2+\|A^*\|_2)(h-1)\delta_{\mathrm{pl}}
},
\]
then there exists \(\bar\eta_k>0\) such that, for every
\(0<\eta<\bar\eta_k\),
\[
\rho\!\left(DT_k^{(\ell)}(\theta^{\NE})\right)<1,
\qquad
T_k^{(\ell)}(\theta)
:=
\theta-\eta(I-\tau\overline A_k^{(\ell)})F(\theta).
\]
\end{corollary}

\begin{proof}
Let \(\Delta_k^{(\ell)}:=\overline A_k^{(\ell)}-A^*\). By the
limited-memory version of Lemma~\ref{lem:app-multilrsga-block},
\[
\|\Delta_k^{(\ell)}\|_2\le(h-1)\delta_{\mathrm{pl}}.
\]
Moreover,
\[
\|S^*\Delta_k^{(\ell)}-\Delta_k^{(\ell)}S^*\|_2
\le
\|S^*\Delta_k^{(\ell)}\|_2
+
\|\Delta_k^{(\ell)}S^*\|_2
\le
2\|S^*\|_2\|\Delta_k^{(\ell)}\|_2,
\]
and similarly
\[
\|\Delta_k^{(\ell)}A^*+A^*\Delta_k^{(\ell)}\|_2
\le
\|\Delta_k^{(\ell)}A^*\|_2
+
\|A^*\Delta_k^{(\ell)}\|_2
\le
2\|A^*\|_2\|\Delta_k^{(\ell)}\|_2.
\]
Therefore,
\[
\begin{aligned}
&\|S^*\Delta_k^{(\ell)}-\Delta_k^{(\ell)}S^*\|_2
+
\|\Delta_k^{(\ell)}A^*+A^*\Delta_k^{(\ell)}\|_2 \\
&\qquad\le
2(\|S^*\|_2+\|A^*\|_2)(h-1)\delta_{\mathrm{pl}} .
\end{aligned}
\]
Hence the stated condition on \(\tau\) implies the condition in
Proposition~\ref{prop:app-frozen-stability} with
\(\mathcal A_k=\overline A_k^{(\ell)}\), and the claim follows.
\end{proof}

We will also use the following immediate variant, which only requires a uniform
bound on \(\|\Delta_k^{(\ell)}\|_2\).

\begin{corollary}[Global \(\tau\)-bound under skew-error boundedness]
\label{cor:app-delta-tau}
Let \(\Delta_k^{(\ell)}:=\overline A_k^{(\ell)}-A^*\) and assume that
\(\|\Delta_k^{(\ell)}\|_2\le\delta\) for all \(k\). If
\[
0<\tau<
\frac{2\lambda_{\min}(S^*)}
{
\|S^*A^*-A^*S^*\|_2
+
2(\|S^*\|_2+\|A^*\|_2)\delta
},
\]
then there exists \(\bar\eta_k>0\) such that, for every
\(0<\eta<\bar\eta_k\),
\[
\rho\!\left(DT_k^{(\ell)}(\theta^{\NE})\right)<1 .
\]
\end{corollary}

\begin{proof}
The proof is the same as the proof of
Corollary~\ref{cor:app-playerwise-tau}, using directly
\(\|\Delta_k^{(\ell)}\|_2\le\delta\).
\end{proof}

\paragraph{Uniform boundedness of the LM--MultiLRSGA correction.}
We now adapt Lemma 5.1 and Corollary 5.2 of \citep{foglia2026lmlrsga} to the
\(h\)-player setting. The next result combines both statements.

\begin{lemma}[Uniform boundedness of the limited-memory matrices and skew correction]
\label{lem:app-lm-boundedness}
Let \(\Omega\subset\mathbb R^d\) be convex and assume that, for each
\(i=1,\dots,h\), the map
\(\partial_{\theta_i}f_i:\Omega\to\mathbb R^{d_i}\) is Lipschitz continuous with
constant \(L_i>0\). Fix \(\ell\ge1\) and assume that all pairs entering either the active window or the rank-one base
satisfy \(s^t\neq0\). Then, for all \(i=1,\dots,h\),
\[
\|M_{i,k}^{(\ell)}\|_2
=
\|(M_{i,k}^{(\ell)})^\top\|_2
\le
C_i(\ell):=(\ell+1)L_i .
\]
Moreover, setting $C_{\max}(\ell):=\max_{i=1,\dots,h}C_i(\ell)$ and $C_A(\ell):=(h-1)C_{\max}(\ell)$, one has
\[
\|\overline A_k^{(\ell)}\|_2\le C_A(\ell).
\]
Consequently, with \(\Delta_k^{(\ell)}:=\overline A_k^{(\ell)}-A^*\), it holds $\|\Delta_k^{(\ell)}\|_2 \le \delta_A(\ell):=C_A(\ell)+\|A^*\|_2$

\end{lemma}

\begin{proof}
For each stored pair, $y_i^t$ and $s^t$, since \(\partial_{\theta_i}f_i\) is \(L_i\)-Lipschitz on \(\Omega\), it holds $\|y_i^t\|_2\le L_i\|s^t\|_2$.

By definition \(p_t=((s^t)^\top s^t)^{-1}=\|s^t\|_2^{-2}\), and then each rank-one term satisfies
\(
\|p_t\,y_i^t(s^t)^\top\|_2
=
p_t\|y_i^t\|_2\|s^t\|_2
\le
L_i .
\)
Furthermore, \(V_t=I-p_t s^t(s^t)^\top\) it is easy verify that \(\|V_t\|_2\le1\). The initialization
\(H_{0,i}^k\) is of the same rank-one form, so
\(
\|H_{0,i}^k\|_2
=
\frac{\|y_i^{k-\ell-1}\|_2\|s^{k-\ell-1}\|_2}
{\|s^{k-\ell-1}\|_2^2}
\le L_i .
\)
Therefore, by the limited-memory expansion and the triangle inequality,
\[
\begin{aligned}
\|M_{i,k}^{(\ell)}\|_2
&\le
\|H_{0,i}^k\|_2
\|V_{k-\ell}\cdots V_{k-1}\|_2 +
\sum_{t=k-\ell}^{k-1}
\|p_t\,y_i^t(s^t)^\top\|_2
\|V_{t+1}\cdots V_{k-1}\|_2  \\
&\le
L_i+\sum_{t=k-\ell}^{k-1}L_i
=
(\ell+1)L_i .
\end{aligned}
\]
The transpose bound follows from
\(\|(M_{i,k}^{(\ell)})^\top\|_2=\|M_{i,k}^{(\ell)}\|_2\).
It remains to bound \(\overline A_k^{(\ell)}\). 
Since each block has norm bounded by the norm of the full rectangular matrix,
\(
\|[M_{i,k}^{(\ell)}]_j\|_2\le\|M_{i,k}^{(\ell)}\|_2\le C_i(\ell),
\)
and similarly
\(
\|([M_{j,k}^{(\ell)}]_i)^\top\|_2
=
\|[M_{j,k}^{(\ell)}]_i\|_2
\le
C_j(\ell),
\)
and then, for \(i\neq j\),
\[
\|[\overline A_k^{(\ell)}]_{ij}\|_2
= \frac12 \Big|\Big| 
\Bigl([M_{i,k}^{(\ell)}]_j-([M_{j,k}^{(\ell)}]_i)^\top\Bigr)\Big|\Big|_2
\le
\frac12(C_i(\ell)+C_j(\ell))
\le
C_{\max}(\ell).
\]
Let \(v=(v_1,\dots,v_h)\in\mathbb R^d\), with
\(v_i\in\mathbb R^{d_i}\). Since the diagonal blocks of
\(\overline A_k^{(\ell)}\) vanish,
\(
(\overline A_k^{(\ell)}v)_i
=
\sum_{j\neq i}[\overline A_k^{(\ell)}]_{ij}v_j .
\)
Thus we obtain the following inequality
\[
\begin{aligned}
\|(\overline A_k^{(\ell)}v)_i\|_2
&\le
\sum_{j\neq i}
\|[\overline A_k^{(\ell)}]_{ij}\|_2\|v_j\|_2 
\le
C_{\max}(\ell)
\sum_{j\neq i}\|v_j\|_2 \le
\sqrt{h-1}\,C_{\max}(\ell)
\Bigl(\sum_{j\neq i}\|v_j\|_2^2\Bigr)^{1/2}.
\end{aligned}
\]
Squaring and summing over \(i\) gives
\[
\begin{aligned}
\|\overline A_k^{(\ell)}v\|_2^2
&=
\sum_{i=1}^h
\|(\overline A_k^{(\ell)}v)_i\|_2^2 \le
(h-1)C_{\max}(\ell)^2
\sum_{i=1}^h\sum_{j\neq i}\|v_j\|_2^2 \\
&=
(h-1)^2C_{\max}(\ell)^2
\sum_{j=1}^h\|v_j\|_2^2 =
(h-1)^2C_{\max}(\ell)^2\|v\|_2^2 .
\end{aligned}
\]
Therefore
\(
\|\overline A_k^{(\ell)}\|_2
\le
(h-1)C_{\max}(\ell)
=
C_A(\ell).
\)
And finally,
\[
\|\Delta_k^{(\ell)}\|_2
=
\|\overline A_k^{(\ell)}-A^*\|_2
\le
\|\overline A_k^{(\ell)}\|_2+\|A^*\|_2
\le
C_A(\ell)+\|A^*\|_2
=
\delta_A(\ell).
\]
\end{proof}

\paragraph{Local convergence.}
The following theorem is the \(h\)-player limited-memory analogue of
Theorem~5.3 in \citep{foglia2026lmlrsga}. We state it for completeness and only
prove the estimates that differ from the two-objective case.

\begin{theorem}[Local linear convergence of LM--MultiLRSGA]
\label{thm:app-lm-multilrsga-convergence}
Let \(\Omega\subset\mathbb R^d\) be a convex open set and assume
\(f_i\in C^3(\Omega,\mathbb R)\) for \(i=1,\dots,h\). Assume that, for each
\(i\), the map \(\partial_{\theta_i}f_i:\Omega\to\mathbb R^{d_i}\) is
Lipschitz continuous on \(\Omega\). Let \(\theta^{\NE}\in\Omega\) be a stable
Nash equilibrium such that
\[
F(\theta^{\NE})=0,
\qquad
H^*=DF(\theta^{\NE})=S^*+A^*,
\qquad
S^*=S^{*\top}\succ0,
\qquad
A^*=-A^{*\top}.
\]
Then, for every fixed history length \(\ell\ge1\), there exist constants $\bar\tau(\ell)>0$, $\bar\eta(\ell)>0$ and $r(\ell)>0$,
such that, if \(0<\tau<\bar\tau(\ell)\), \(0<\eta<\bar\eta(\ell)\), and
\(\theta^0\in B_{r(\ell)}(\theta^{\NE})\), the LM--MultiLRSGA iterates
\[
\theta^{k+1}
=
\theta^k-\eta\bigl(I-\tau\overline A_k^{(\ell)}\bigr)F(\theta^k)
\]
remain in \(B_{r(\ell)}(\theta^{\NE})\) and converge linearly to
\(\theta^{\NE}\). More precisely, there exists \(q(\ell)\in(0,1)\) such that
\[
\|\theta^k-\theta^{\NE}\|_2
\le
q(\ell)^k\|\theta^0-\theta^{\NE}\|_2,
\qquad k\ge0.
\]
\end{theorem}

\begin{proof}
The proof follows the contraction argument of Theorem~5.3 in
\citep{foglia2026lmlrsga}. The only estimates that change are those controlling
the limited-memory skew correction. In the present \(h\)-player setting,
Lemma~\ref{lem:app-lm-boundedness} gives
\[
\|\overline A_k^{(\ell)}\|_2\le C_A(\ell),
\qquad
\|\overline A_k^{(\ell)}-A^*\|_2\le\delta_A(\ell),
\]
uniformly in \(k\), with the factor \((h-1)\) encoded in \(C_A(\ell)\). Hence
Corollary~\ref{cor:app-delta-tau} gives a \(\bar\tau(\ell)>0\) such that the
frozen matrices \(G_k^{(\ell)}=(I-\tau\overline A_k^{(\ell)})H^*\) satisfy the
same uniform coercivity estimate used in \citep{foglia2026lmlrsga}. The bound
\[
\|G_k^{(\ell)}\|_2\le(1+\tau C_A(\ell))\|H^*\|_2
\]
then yields a common \(\bar\eta(\ell)>0\) for which the linearized frozen maps
are uniformly contractive. The remaining Taylor-remainder argument is unchanged:
for \(f_i\in C^3(\Omega)\), \(F\) has a quadratic remainder around
\(\theta^{\NE}\), and choosing \(r(\ell)>0\) sufficiently small gives invariance
of \(B_{r(\ell)}(\theta^{\NE})\) and linear convergence.
\end{proof}

The theorem shows that, for every fixed memory length \(\ell\), LM--MultiLRSGA
inherits the local linear convergence mechanism of LM--LRSGA. The passage from
two objectives to \(h\) players affects the constants through the block-skew
estimate, in particular through the \((h-1)\) factor appearing in the control of
\(\overline A_k^{(\ell)}\) and in the MultiLRSGA block estimate.

\clearpage
\newpage

\section{Additional PINN diagnostics}
\label{app:pinn_additional_diagnostics}

This appendix reports additional diagnostics for the hard-constrained Burgers PINN
experiments in Section~\ref{sec:pinn_experiments}. Table \ref{tab:table1-minmax} shows all metrics for the main experiments. Also, we show additional figures that are not reported as main evidence in the paper; they document the behavior of the individual
residual components and of the LM--MultiLRSGA phase.


\begin{table*}[hbt!]
\centering
\caption{
Test losses for the hard-constrained Burgers PINN benchmark. The results of Table \ref{tab:pinn_exp1_test_losses} are extended with min and max computed on the ten seeds.
Lower values are better. The column $\Delta$ contains the percentage difference relative to the best value achieved for $\Lsum$.
}
\label{tab:table1-minmax}
\resizebox{\textwidth}{!}{%
\begin{tabular}{llccccccccccccc}
\toprule
 &  & \multicolumn{3}{c}{\(\Lleft\)} & \multicolumn{3}{c}{\(\Lcenter\)} & \multicolumn{3}{c}{\(\Lright\)} & \multicolumn{3}{c}{\(\Lsum\)} &  \\
\cmidrule(lr){3-5}\cmidrule(lr){6-8}\cmidrule(lr){9-11}\cmidrule(lr){12-14}
LR & Method & Mean \(\pm\) std & Min & Max & Mean \(\pm\) std & Min & Max & Mean \(\pm\) std & Min & Max & Mean \(\pm\) std & Min & Max & $\Delta$ (best) \\
\midrule
\(10^{-2}\) & \methodours & \(0.1859\pm0.0705\) & \(0.1165\) & \(0.3557\) & \(0.8123\pm0.0292\) & \(0.7636\) & \(0.8587\) & \(0.1551\pm0.0633\) & \(0.0909\) & \(0.2931\) & \(1.1533\pm0.0797\) & \(1.0422\) & \(1.3184\) & \(0.0\%\) \\
\(10^{-2}\) & PCGrad & \(0.2455\pm0.0855\) & \(0.1353\) & \(0.4123\) & \(0.8515\pm0.0382\) & \(0.8023\) & \(0.9252\) & \(0.1180\pm0.0158\) & \(0.0947\) & \(0.1419\) & \(1.2151\pm0.0846\) & \(1.1036\) & \(1.3956\) & \(-5.4\%\) \\
\(10^{-2}\) & MultiAdam & \(0.4082\pm0.2926\) & \(0.1457\) & \(1.0622\) & \(0.8028\pm0.1645\) & \(0.5956\) & \(1.1556\) & \(0.5137\pm0.2062\) & \(0.2516\) & \(0.8429\) & \(1.7247\pm0.3459\) & \(1.3142\) & \(2.5864\) & \(-49.5\%\) \\
\(10^{-2}\) & DualCone center & \(0.2012\pm0.0829\) & \(0.1166\) & \(0.3250\) & \(0.8511\pm0.0368\) & \(0.8023\) & \(0.9131\) & \(0.2122\pm0.0772\) & \(0.1257\) & \(0.3417\) & \(1.2646\pm0.0974\) & \(1.1526\) & \(1.4572\) & \(-9.6\%\) \\
\(10^{-2}\) & DualCone avg & \(0.3151\pm0.0564\) & \(0.2373\) & \(0.3892\) & \(0.9517\pm0.0346\) & \(0.8955\) & \(1.0042\) & \(0.3161\pm0.0945\) & \(0.2260\) & \(0.5075\) & \(1.5829\pm0.1589\) & \(1.3588\) & \(1.8520\) & \(-37.2\%\) \\
\(10^{-2}\) & DualCone proj & \(0.3129\pm0.0575\) & \(0.2351\) & \(0.3829\) & \(0.9444\pm0.0413\) & \(0.8589\) & \(1.0014\) & \(0.3147\pm0.0927\) & \(0.2311\) & \(0.5024\) & \(1.5720\pm0.1621\) & \(1.3251\) & \(1.8372\) & \(-36.3\%\) \\
\midrule
\(10^{-3}\) & \methodours & \(0.2509\pm0.1056\) & \(0.1124\) & \(0.4127\) & \(0.8529\pm0.0350\) & \(0.8091\) & \(0.9424\) & \(0.1889\pm0.1144\) & \(0.0857\) & \(0.3833\) & \(1.2927\pm0.0725\) & \(1.1558\) & \(1.3868\) & \(0.0\%\) \\
\(10^{-3}\) & PCGrad & \(0.5248\pm0.0993\) & \(0.3363\) & \(0.6509\) & \(0.9604\pm0.0471\) & \(0.8647\) & \(1.0131\) & \(0.3878\pm0.1307\) & \(0.2348\) & \(0.5713\) & \(1.8730\pm0.1672\) & \(1.6644\) & \(2.1226\) & \(-44.9\%\) \\
\(10^{-3}\) & MultiAdam & \(0.5160\pm0.3818\) & \(0.1900\) & \(1.4165\) & \(0.7056\pm0.1096\) & \(0.5381\) & \(0.8477\) & \(0.7054\pm0.5435\) & \(0.1887\) & \(1.6261\) & \(1.9270\pm0.7098\) & \(1.1520\) & \(2.9862\) & \(-49.1\%\) \\
\(10^{-3}\) & DualCone center & \(0.5651\pm0.1765\) & \(0.3175\) & \(0.9084\) & \(0.9303\pm0.0506\) & \(0.8124\) & \(0.9920\) & \(0.4774\pm0.1559\) & \(0.2901\) & \(0.7597\) & \(1.9727\pm0.1857\) & \(1.7491\) & \(2.2363\) & \(-52.6\%\) \\
\(10^{-3}\) & DualCone avg & \(0.6602\pm0.2543\) & \(0.3334\) & \(1.2644\) & \(0.9649\pm0.0558\) & \(0.8371\) & \(1.0206\) & \(0.5529\pm0.1964\) & \(0.3047\) & \(0.9604\) & \(2.1779\pm0.2774\) & \(1.9130\) & \(2.7912\) & \(-68.5\%\) \\
\(10^{-3}\) & DualCone proj & \(0.6625\pm0.2535\) & \(0.3421\) & \(1.2645\) & \(0.9622\pm0.0585\) & \(0.8277\) & \(1.0203\) & \(0.5558\pm0.1941\) & \(0.3121\) & \(0.9604\) & \(2.1805\pm0.2758\) & \(1.9166\) & \(2.7912\) & \(-68.7\%\) \\
\midrule
\(10^{-4}\) & \methodours & \(0.2694\pm0.1145\) & \(0.1139\) & \(0.4315\) & \(0.8657\pm0.0392\) & \(0.8243\) & \(0.9712\) & \(0.1985\pm0.1247\) & \(0.0863\) & \(0.4019\) & \(1.3336\pm0.0625\) & \(1.2125\) & \(1.4141\) & \(0.0\%\) \\
\(10^{-4}\) & PCGrad & \(0.7573\pm0.3694\) & \(0.3761\) & \(1.7061\) & \(0.9652\pm0.0615\) & \(0.8279\) & \(1.0286\) & \(0.6161\pm0.2273\) & \(0.3242\) & \(1.0968\) & \(2.3385\pm0.3997\) & \(2.0061\) & \(3.3427\) & \(-75.3\%\) \\
\(10^{-4}\) & MultiAdam & \(0.7607\pm0.2909\) & \(0.3428\) & \(1.3473\) & \(0.6634\pm0.0206\) & \(0.6273\) & \(0.6969\) & \(0.6620\pm0.2603\) & \(0.2533\) & \(1.1010\) & \(2.0861\pm0.2612\) & \(1.7262\) & \(2.5397\) & \(-56.4\%\) \\
\(10^{-4}\) & DualCone center & \(0.7700\pm0.3910\) & \(0.3776\) & \(1.7783\) & \(0.9621\pm0.0619\) & \(0.8225\) & \(1.0236\) & \(0.6312\pm0.2307\) & \(0.3304\) & \(1.1166\) & \(2.3633\pm0.4113\) & \(2.0200\) & \(3.4031\) & \(-77.2\%\) \\
\(10^{-4}\) & DualCone avg & \(0.7878\pm0.4165\) & \(0.3806\) & \(1.8735\) & \(0.9678\pm0.0644\) & \(0.8255\) & \(1.0412\) & \(0.6452\pm0.2392\) & \(0.3337\) & \(1.1519\) & \(2.4008\pm0.4513\) & \(2.0300\) & \(3.5593\) & \(-80.0\%\) \\
\(10^{-4}\) & DualCone proj & \(0.7881\pm0.4164\) & \(0.3821\) & \(1.8735\) & \(0.9675\pm0.0647\) & \(0.8246\) & \(1.0412\) & \(0.6456\pm0.2389\) & \(0.3338\) & \(1.1519\) & \(2.4012\pm0.4511\) & \(2.0300\) & \(3.5593\) & \(-80.1\%\) \\
\bottomrule
\end{tabular}
}
\end{table*}

\clearpage
\newpage
\subsection{Local residual curves for Experiment 1}
\label{app:pinn_exp1_local_curves}

Figures~\ref{fig:app_pinn_exp1_local_lr1e2}--\ref{fig:app_pinn_exp1_local_lr1e4}
show the three local residual components \(\Lleft\), \(\Lcenter\), and
\(\Lright\) for the main benchmark. The curves are epoch-averaged mini-batch
losses, consistently with the training-loss convention used in the main text.

\begin{figure}[hbt!]
\centering

\begin{subfigure}{0.65\textwidth}
\centering
\includegraphics[width=\linewidth]{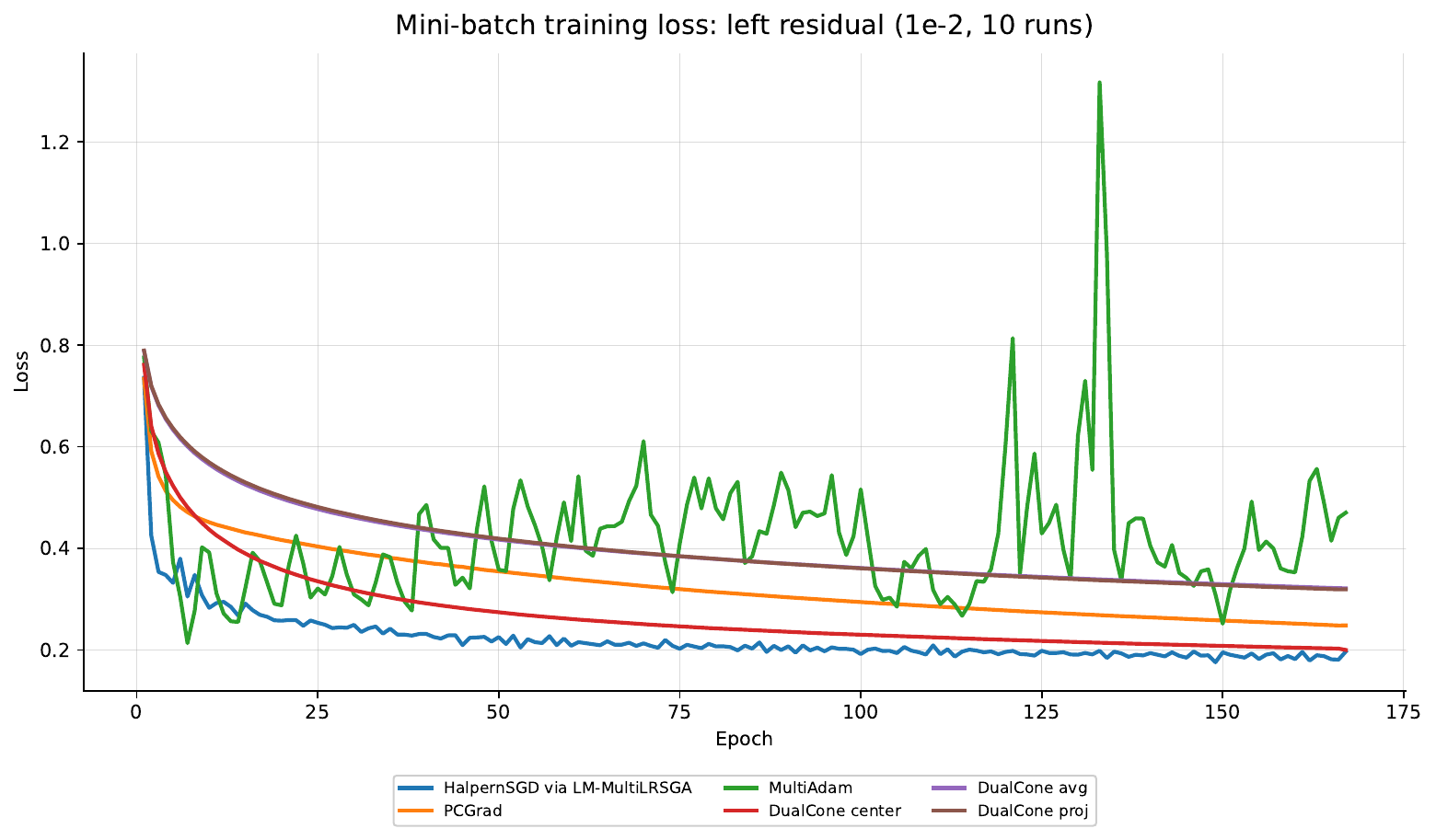}
\caption{\(\Lleft\)}
\end{subfigure}

\medskip

\begin{subfigure}{0.65\textwidth}
\centering
\includegraphics[width=\linewidth]{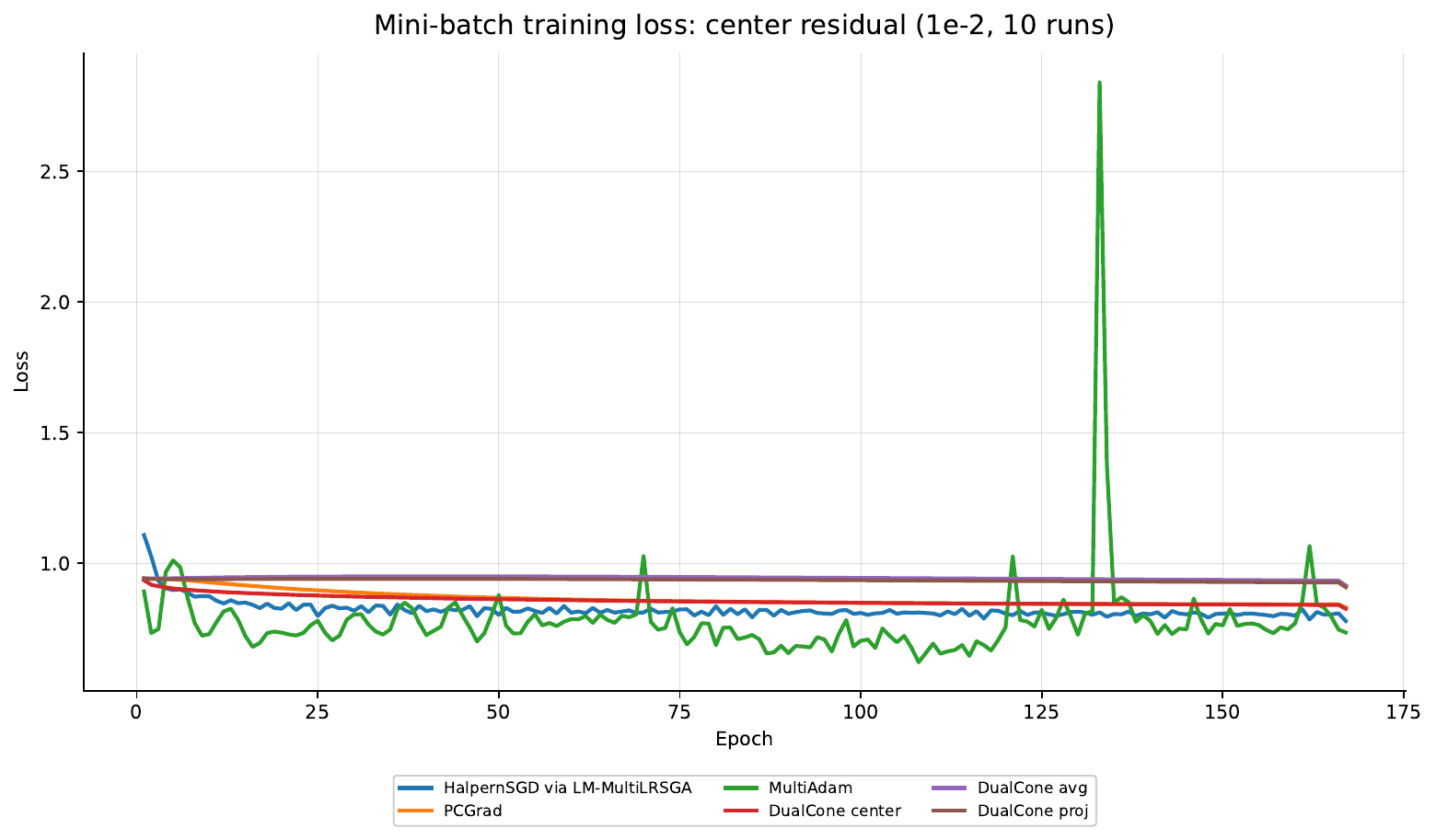}
\caption{\(\Lcenter\)}
\end{subfigure}

\medskip

\begin{subfigure}{0.65\textwidth}
\centering
\includegraphics[width=\linewidth]{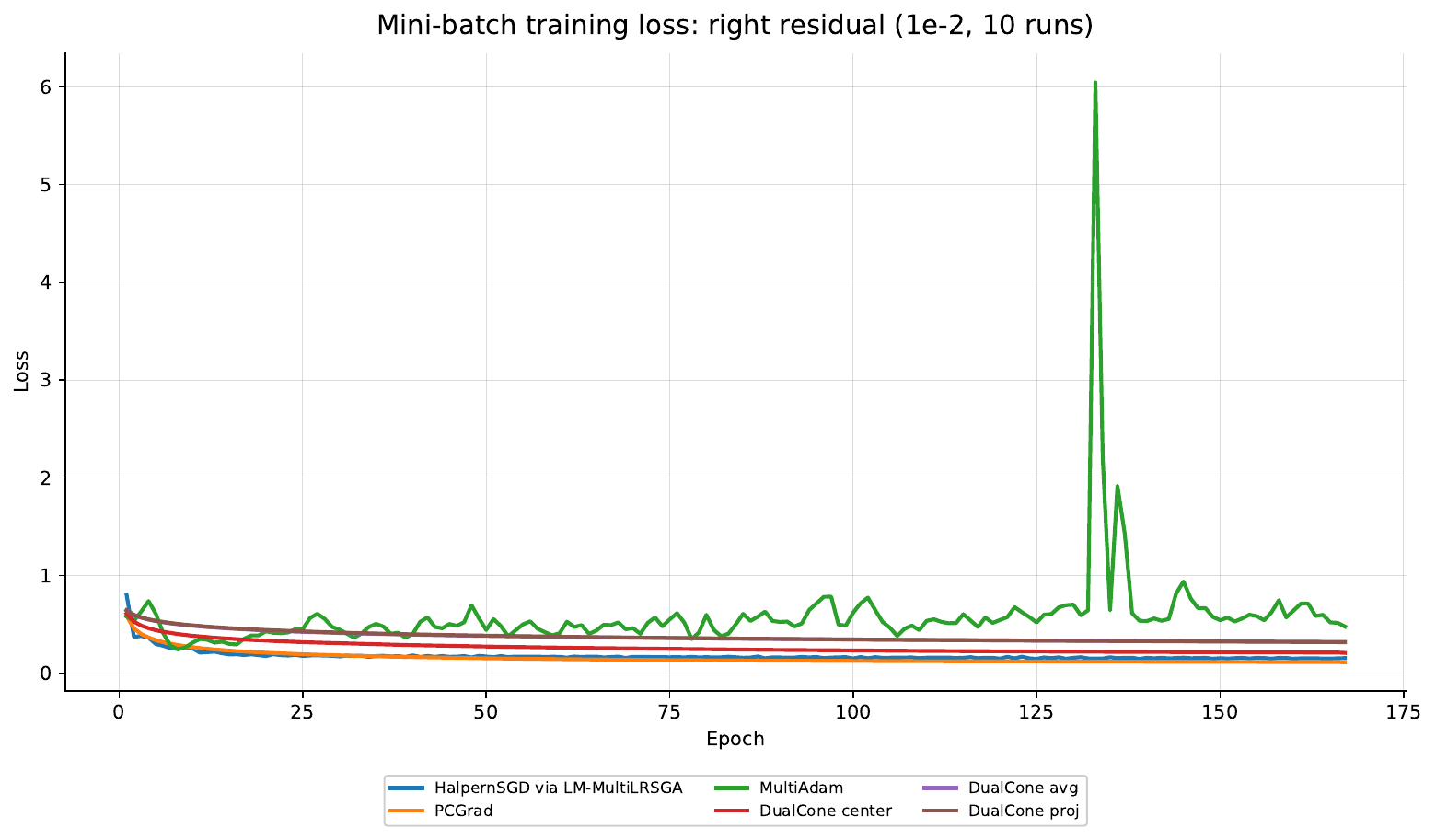}
\caption{\(\Lright\)}
\end{subfigure}

\caption{Experiment~1 local residual curves for \(\eta_0=10^{-2}\).}
\label{fig:app_pinn_exp1_local_lr1e2}
\end{figure}

\begin{figure}[t!]
\centering

\begin{subfigure}{0.65\textwidth}
\centering
\includegraphics[width=\linewidth]{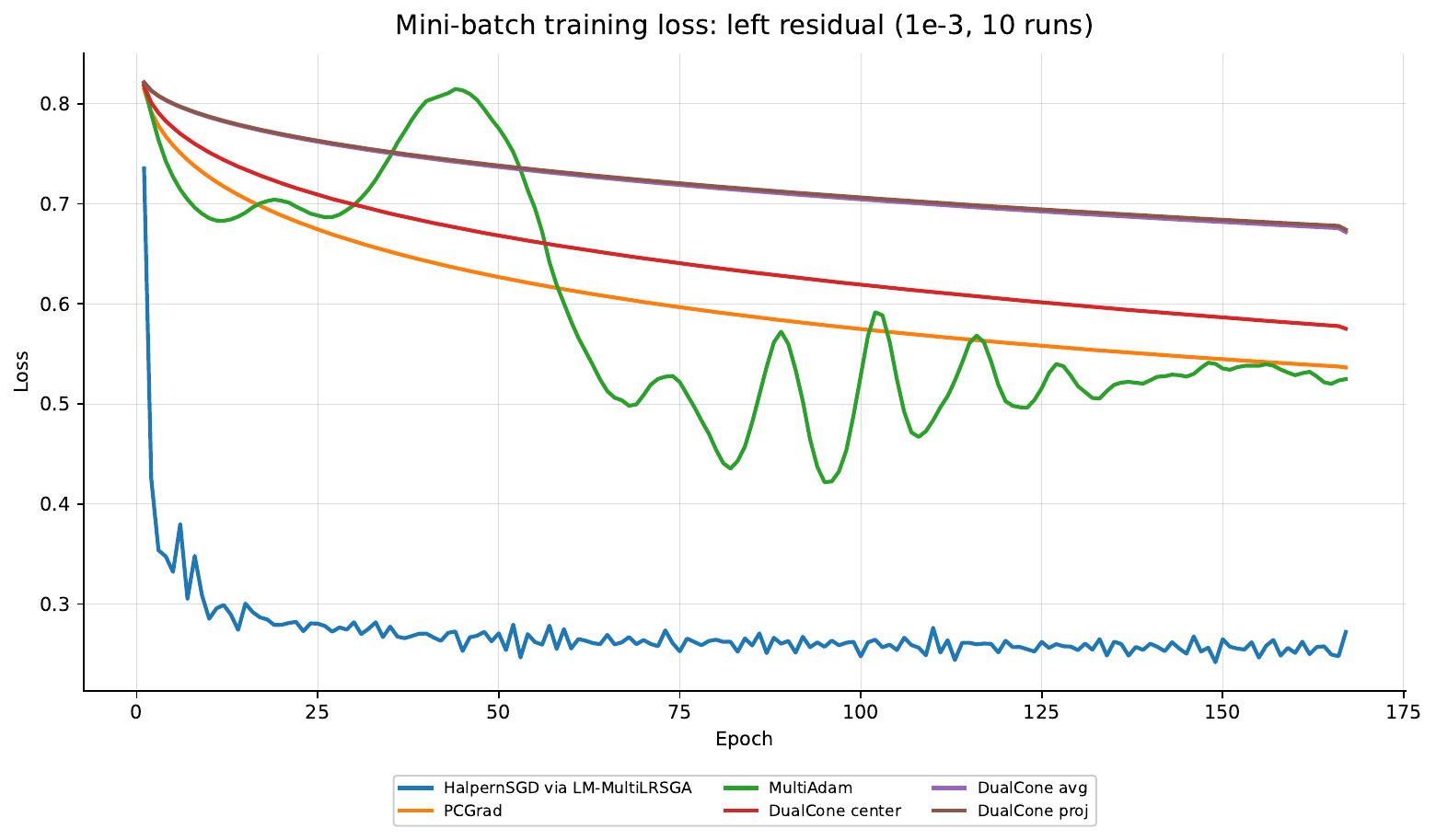}
\caption{\(\Lleft\)}
\end{subfigure}

\medskip

\begin{subfigure}{0.65\textwidth}
\centering
\includegraphics[width=\linewidth]{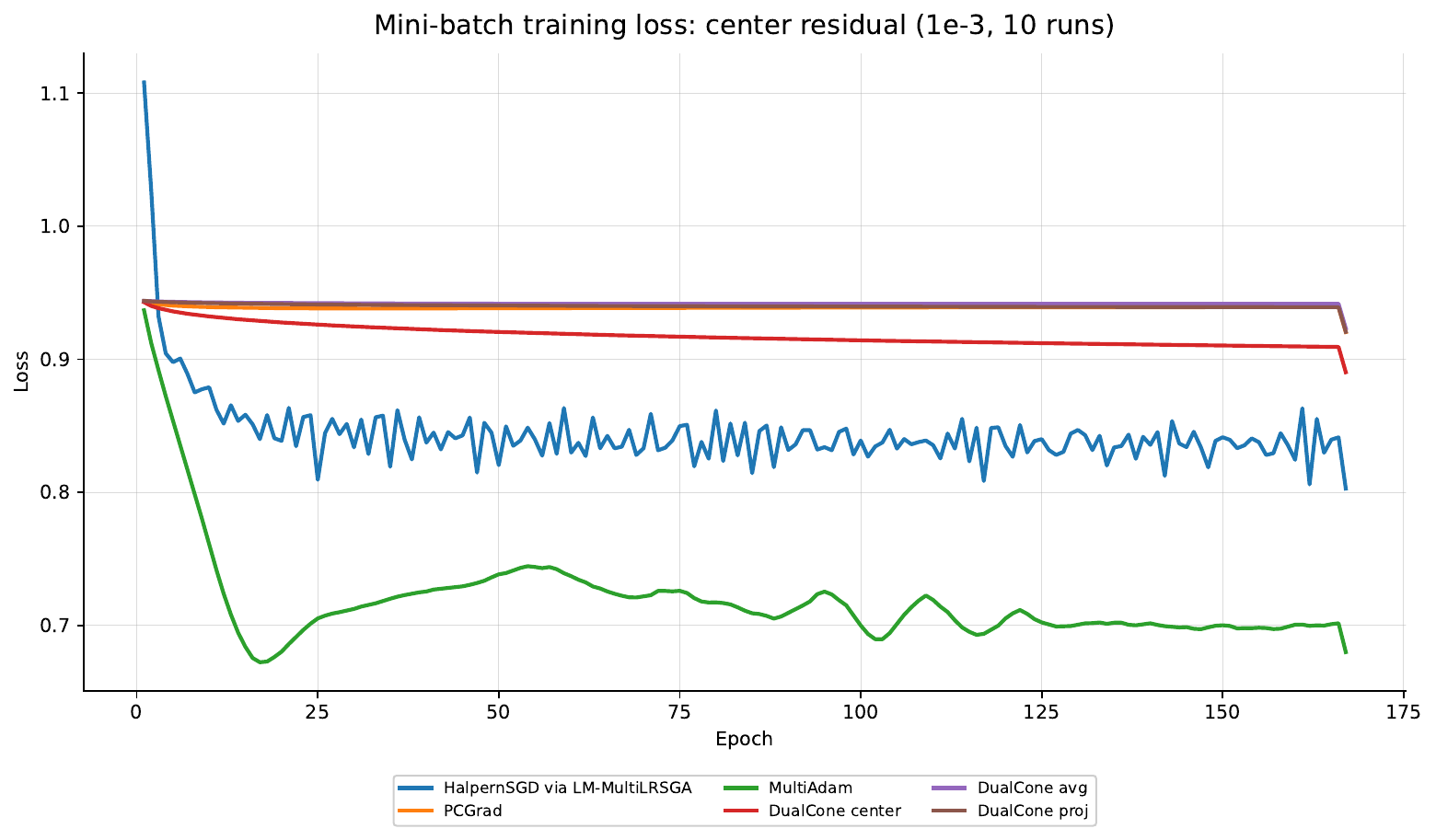}
\caption{\(\Lcenter\)}
\end{subfigure}

\medskip

\begin{subfigure}{0.65\textwidth}
\centering
\includegraphics[width=\linewidth]{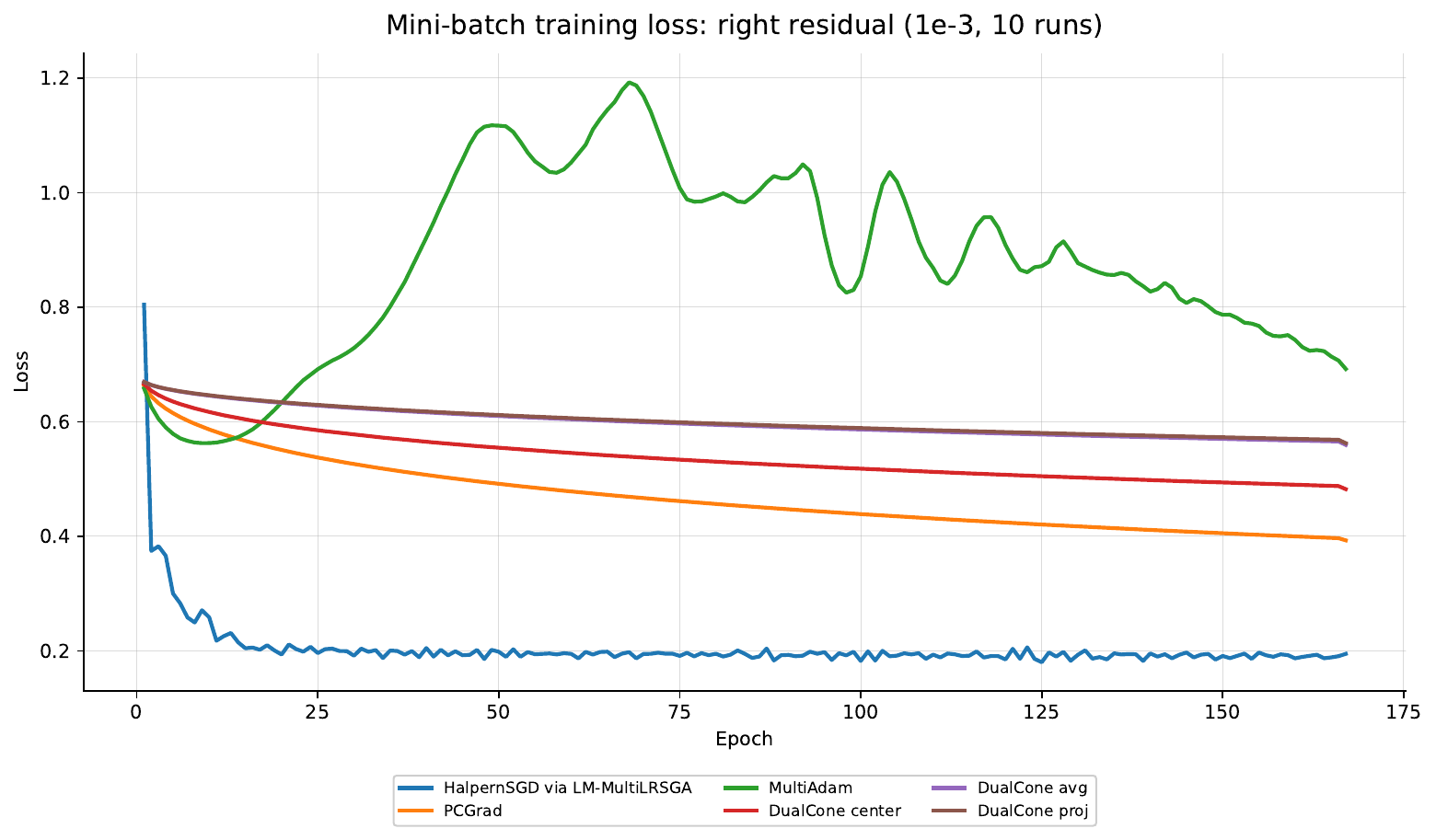}
\caption{\(\Lright\)}
\end{subfigure}

\caption{Experiment~1 local residual curves for \(\eta_0=10^{-3}\).}
\label{fig:app_pinn_exp1_local_lr1e3}
\end{figure}

\begin{figure}[t!]
\centering

\begin{subfigure}{0.75\textwidth}
\centering
\includegraphics[width=\linewidth]{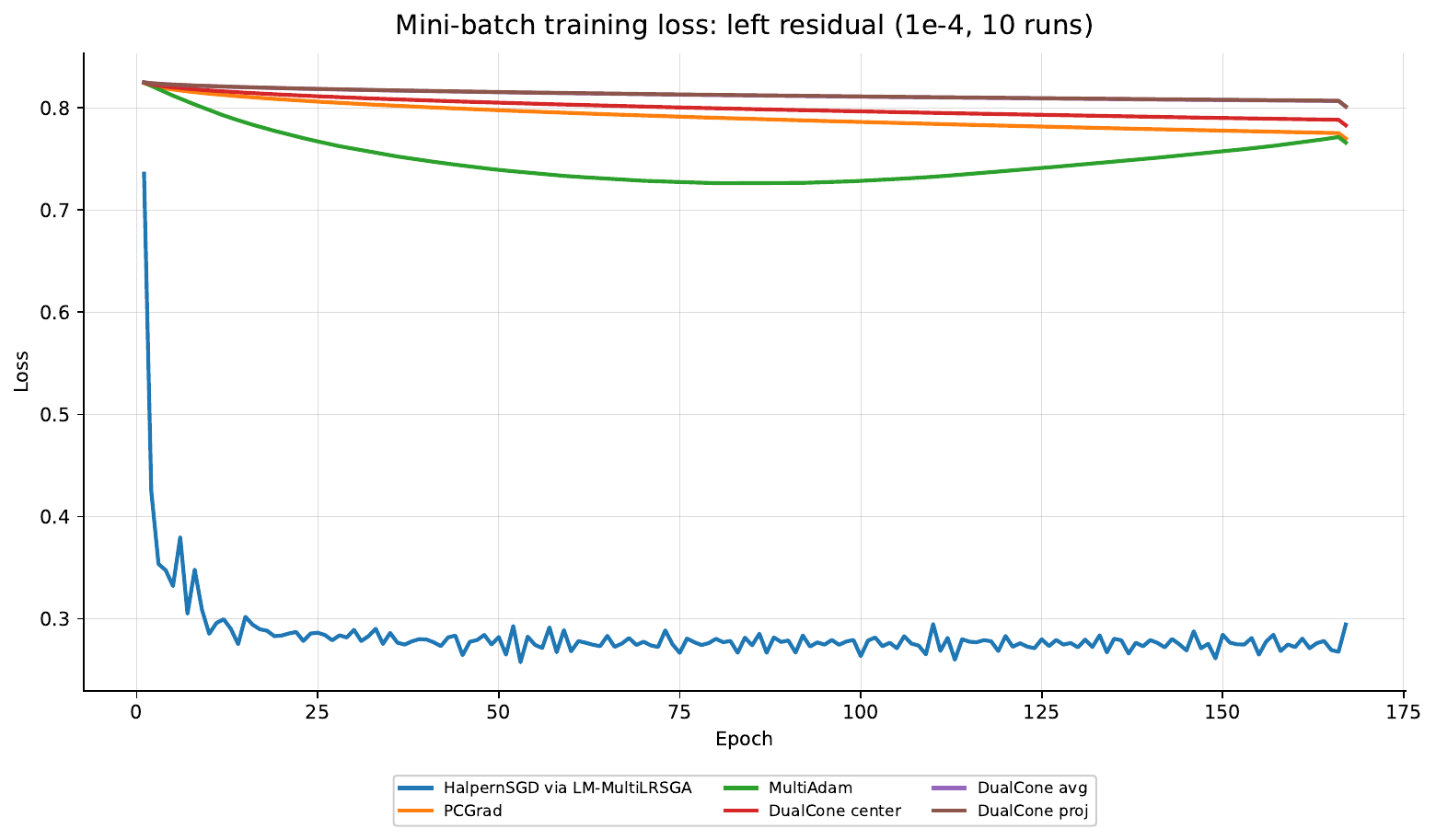}
\caption{\(\Lleft\)}
\end{subfigure}

\medskip

\begin{subfigure}{0.75\textwidth}
\centering
\includegraphics[width=\linewidth]{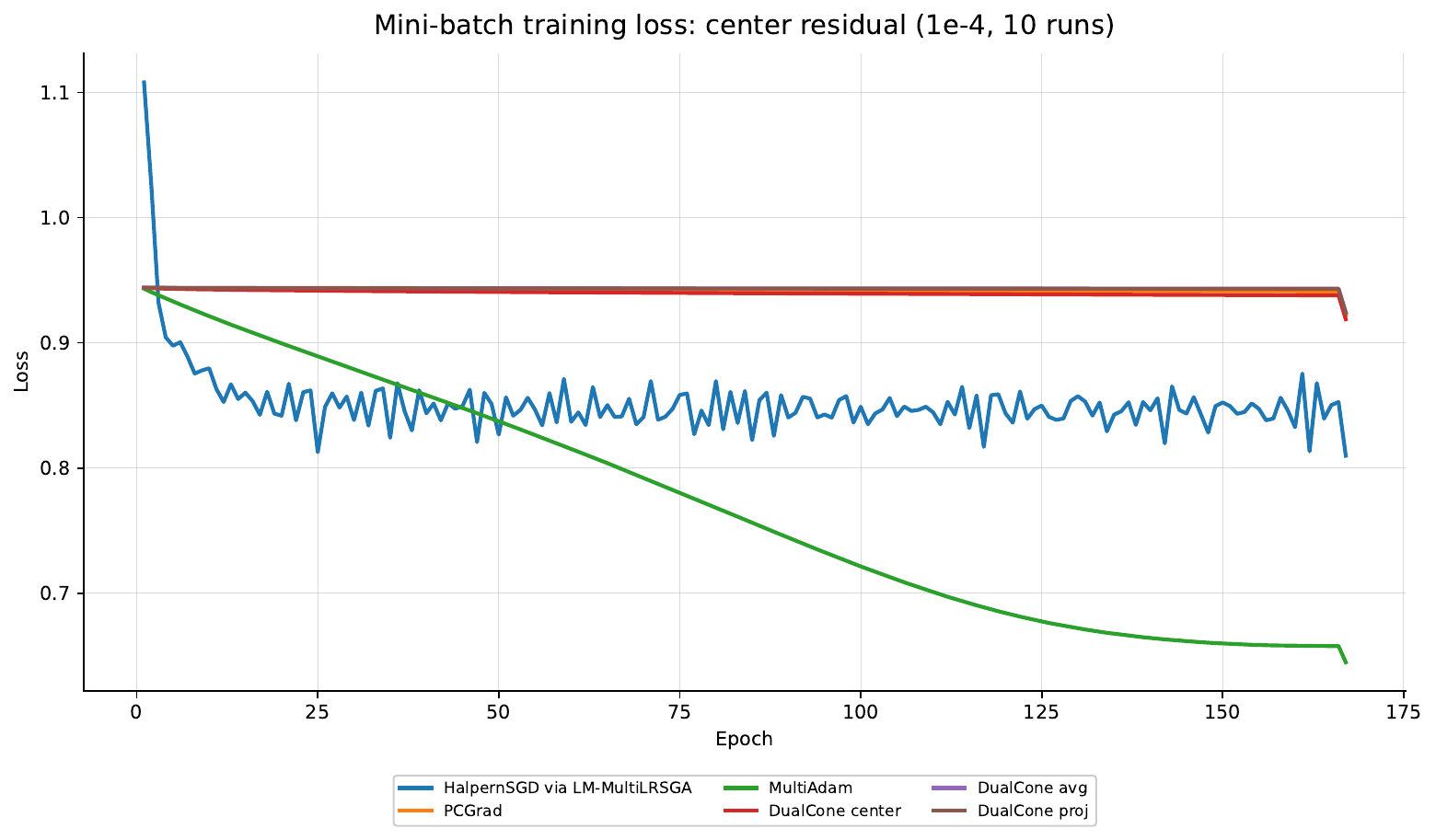}
\caption{\(\Lcenter\)}
\end{subfigure}

\medskip

\begin{subfigure}{0.75\textwidth}
\centering
\includegraphics[width=\linewidth]{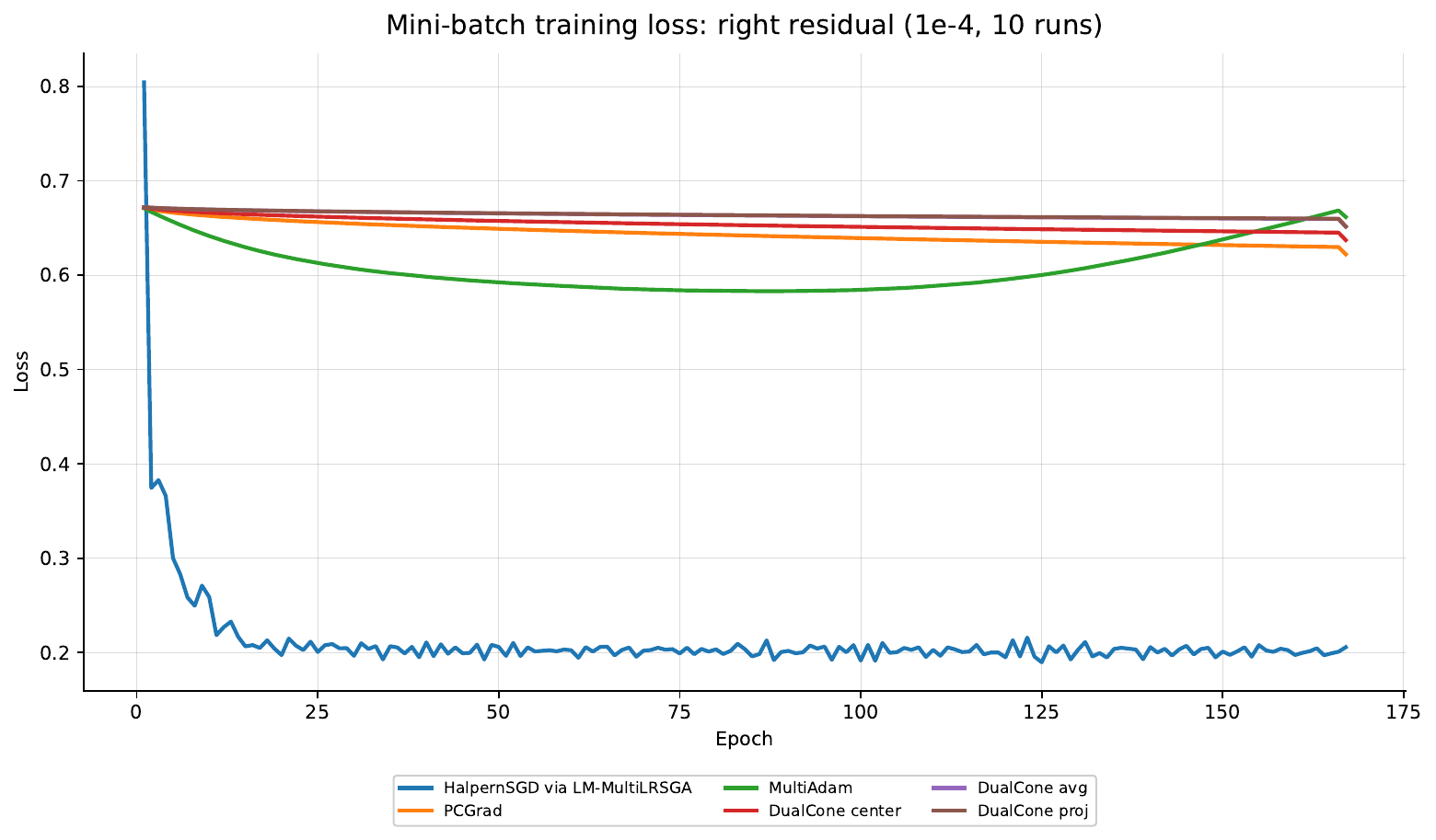}
\caption{\(\Lright\)}
\end{subfigure}

\caption{Experiment~1 local residual curves for \(\eta_0=10^{-4}\).}
\label{fig:app_pinn_exp1_local_lr1e4}
\end{figure}

\clearpage
\newpage
\subsection{Switch-sensitivity diagnostic}
\label{app:pinn_exp2_local_curves}

The switch-sensitivity diagnostic compares different values of the Nash target
that triggers the transition from LM--MultiLRSGA to HalpernSGD. Figure
\ref{fig:app_pinn_exp2_switch_batch_all} reports the epoch-averaged mini-batch
losses. The aggregate full-training curve \(\Lsum\) is shown in the main text;
Figure~\ref{fig:app_pinn_exp2_switch_full_local} reports the corresponding
full-training local components.

\begin{figure}[hbt!]
\centering
\includegraphics[width=0.95\textwidth]{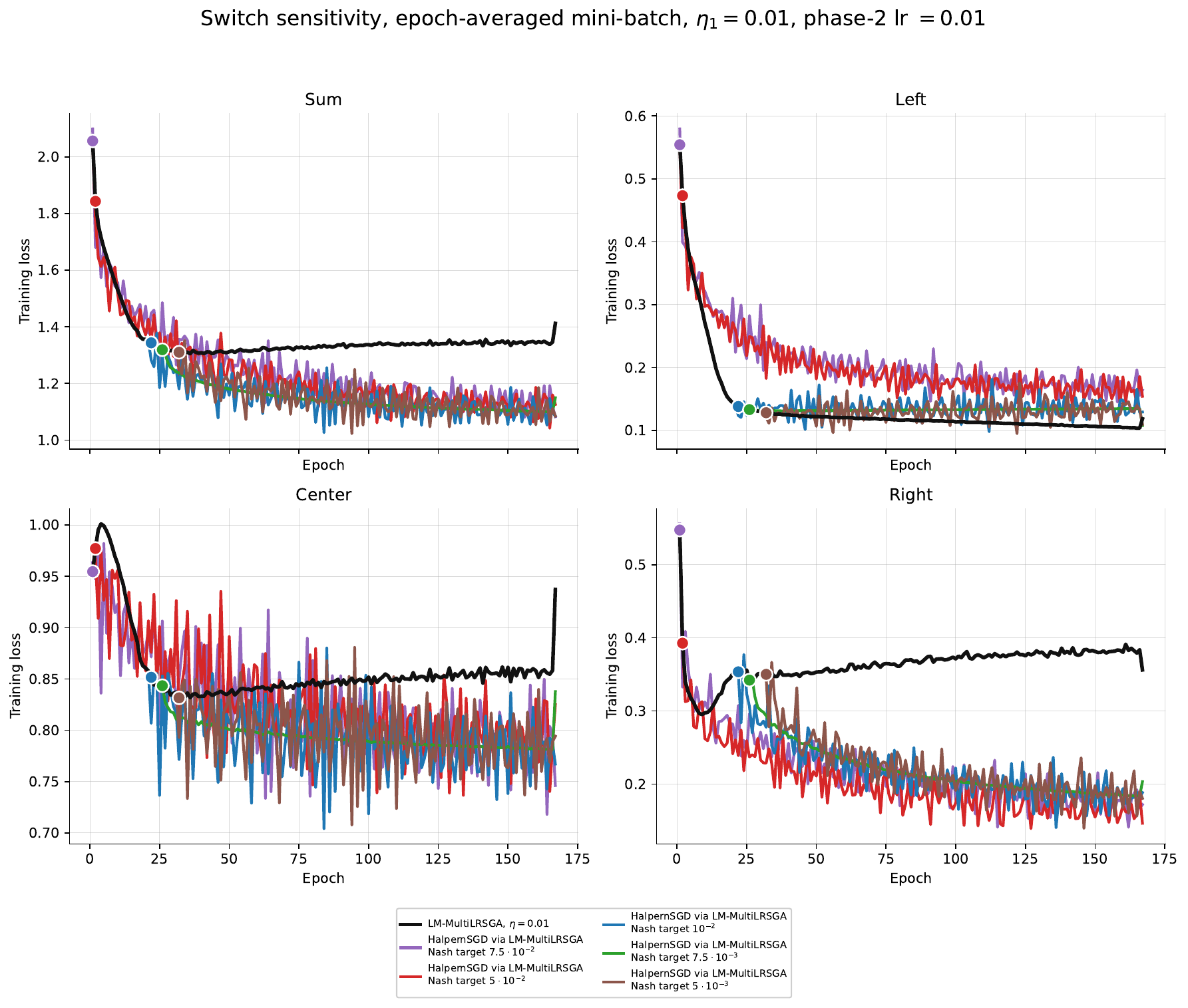}
\caption{
Switch-sensitivity diagnostic with epoch-averaged mini-batch losses. The
LM--MultiLRSGA trajectory is compared with the two-phase method obtained by
switching at different Nash targets.
}
\label{fig:app_pinn_exp2_switch_batch_all}
\end{figure}

\begin{figure}
\centering
\begin{subfigure}{0.32\textwidth}
\centering
\includegraphics[width=\linewidth]{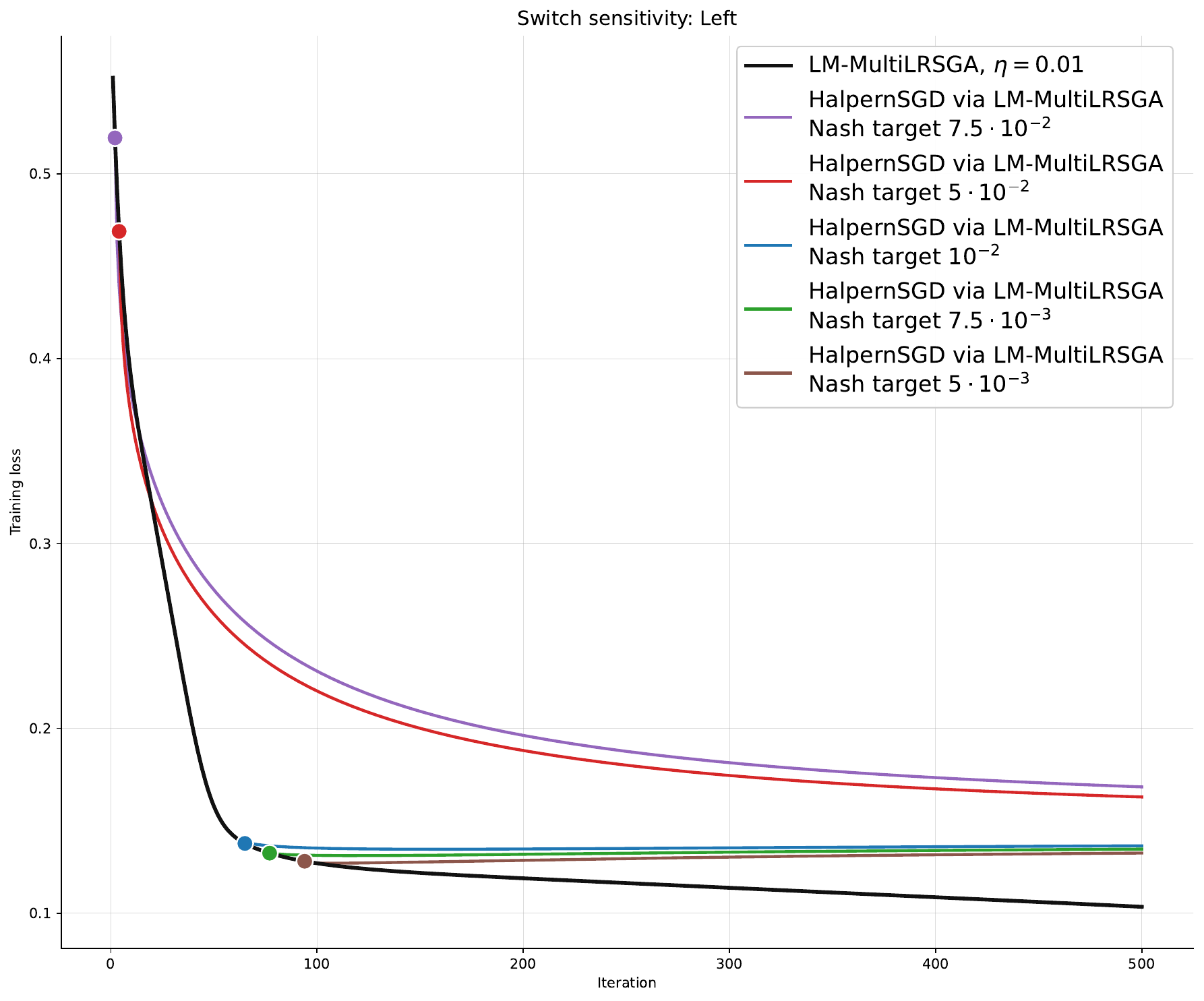}
\caption{\(\Lleft\)}
\end{subfigure}
\hfill
\begin{subfigure}{0.32\textwidth}
\centering
\includegraphics[width=\linewidth]{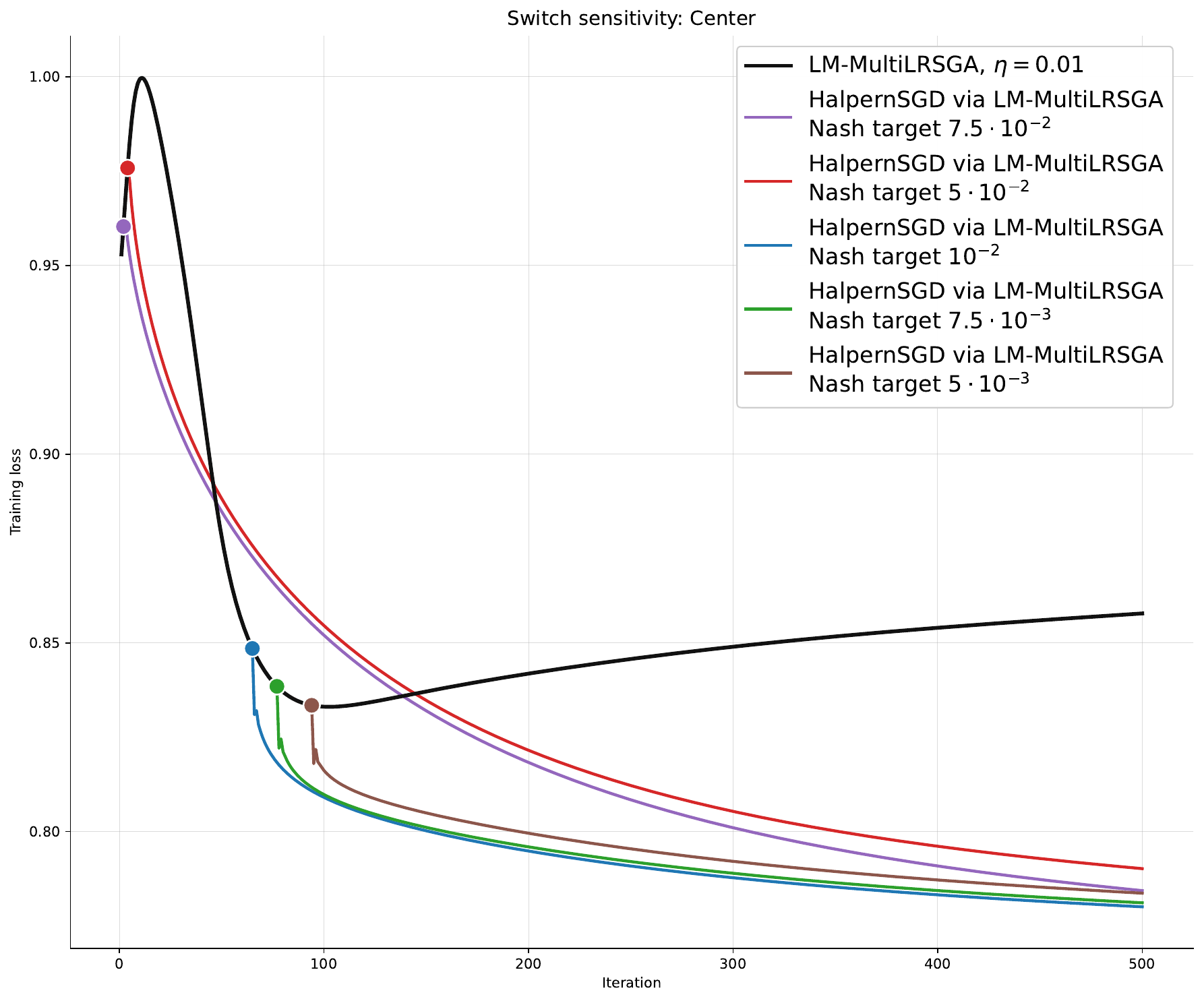}
\caption{\(\Lcenter\)}
\end{subfigure}
\hfill
\begin{subfigure}{0.32\textwidth}
\centering
\includegraphics[width=\linewidth]{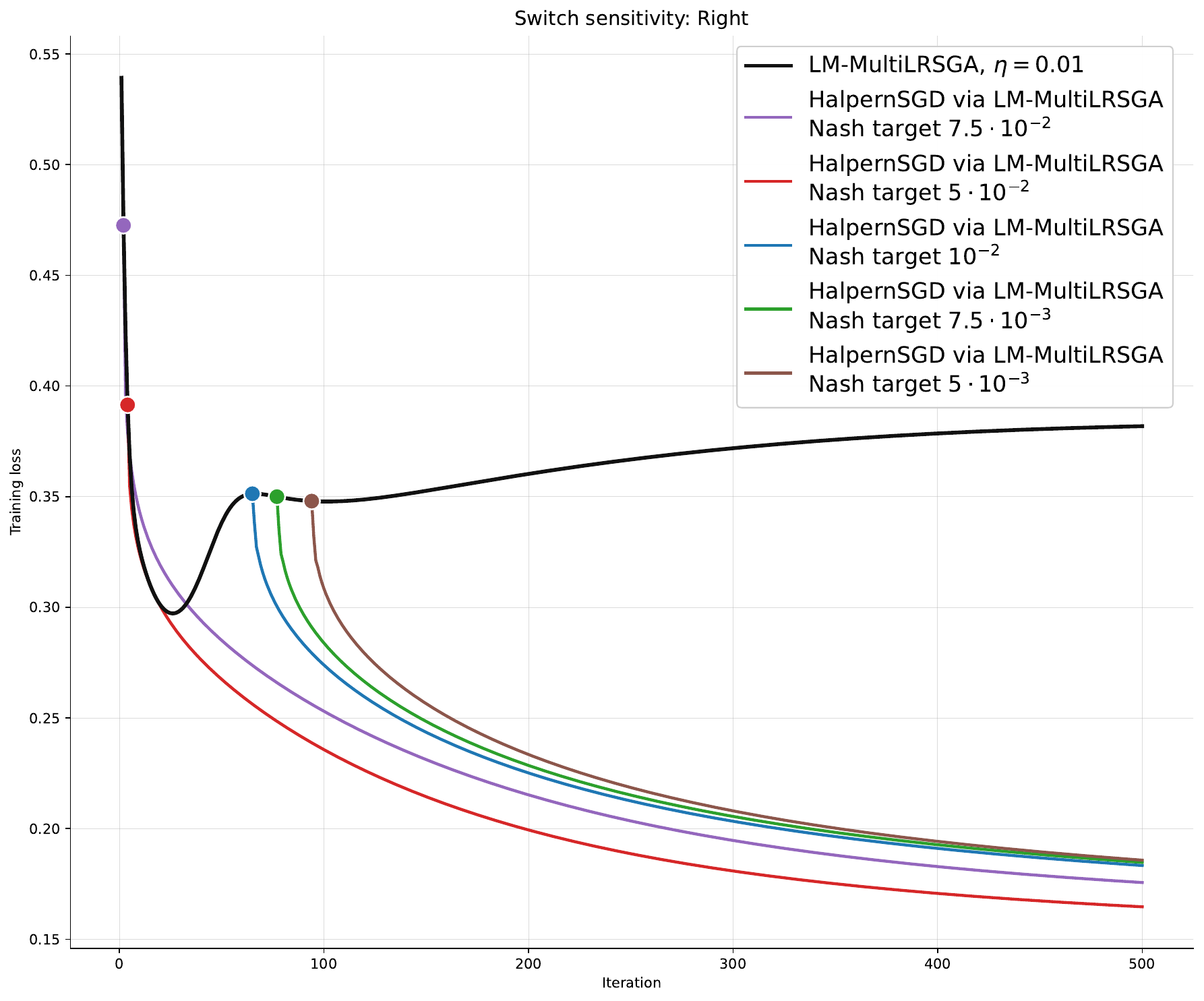}
\caption{\(\Lright\)}
\end{subfigure}
\caption{
Switch-sensitivity diagnostic: full-training local residual components. The
full-training \(\Lsum\) curve is reported in the main text.
}
\label{fig:app_pinn_exp2_switch_full_local}
\end{figure}

\begin{table}[H]
\centering
\caption{
Switch-sensitivity summary on the held-out test collocation grid. The switch
step is the iteration at which the Nash target is first reached.
}
\label{tab:app_pinn_exp2_switch_summary}
\begin{tabular}{lccccc}
\toprule
Nash target & Switch step & \(\Lleft\) & \(\Lcenter\) & \(\Lright\) & \(\Lsum\) \\
\midrule
\(7.5\cdot10^{-2}\) & 2  & 0.1689 & 0.7913 & 0.1749 & 1.1350 \\
\(5\cdot10^{-2}\)   & 4  & 0.1636 & 0.7956 & 0.1655 & 1.1247 \\
\(10^{-2}\)         & 65 & 0.1382 & 0.7863 & 0.1834 & 1.1079 \\
\(7.5\cdot10^{-3}\) & 78 & 0.1365 & 0.7874 & 0.1845 & 1.1085 \\
\(5\cdot10^{-3}\)   & 95 & 0.1343 & 0.7901 & 0.1856 & 1.1100 \\
\bottomrule
\end{tabular}
\end{table}

\clearpage
\newpage
\subsection{LM--MultiLRSGA-only step-size diagnostic}
\label{app:pinn_lm_eta_diagnostic}

This diagnostic studies the competitive phase alone, without the Halpern
refinement. We compare
\begin{equation}
    \eta_{\mathrm{LM}}\in\{0.1,0.05,0.01\},
    \qquad
    \tau_{\mathrm{LM}}=\eta_{\mathrm{LM}}/10,
    \label{eq:app_lm_eta_sweep}
\end{equation}
using the same seed and the first \(200\) LM--MultiLRSGA iterations. The
diagnostic is intended only to interpret the phase-1 behavior; it is not part of
the main baseline comparison.

\begin{figure}[hbt!]
    \centering
    \includegraphics[width=0.75\linewidth]{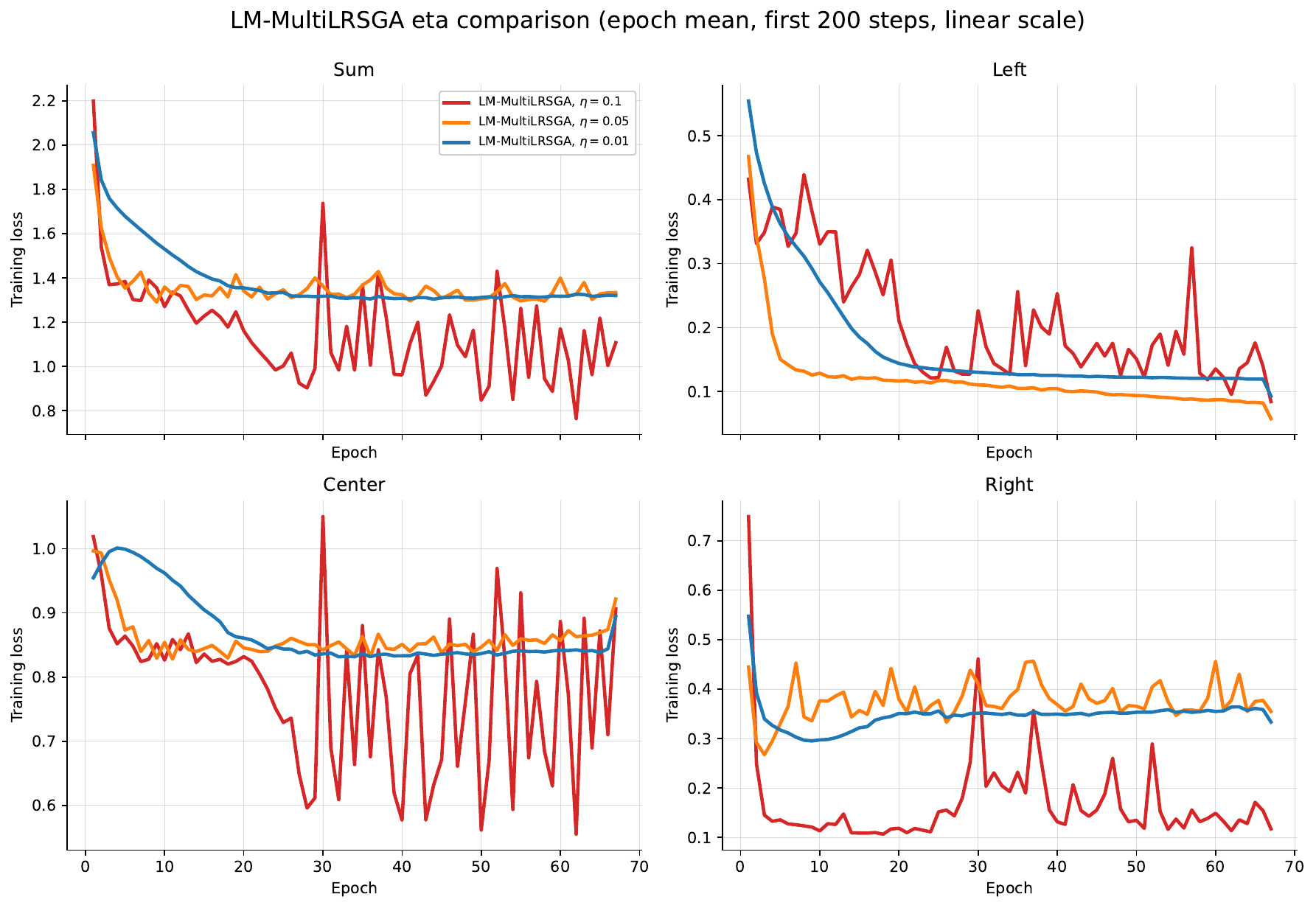}
    \caption{LM--MultiLRSGA-only step-size diagnostic, linear scale, over the first \(200\) iterations.}
    \label{fig:app_pinn_lm_eta_diagnostic_log}
\end{figure}

\begin{figure}[hbt!]
    \centering
    \includegraphics[width=0.75\linewidth]{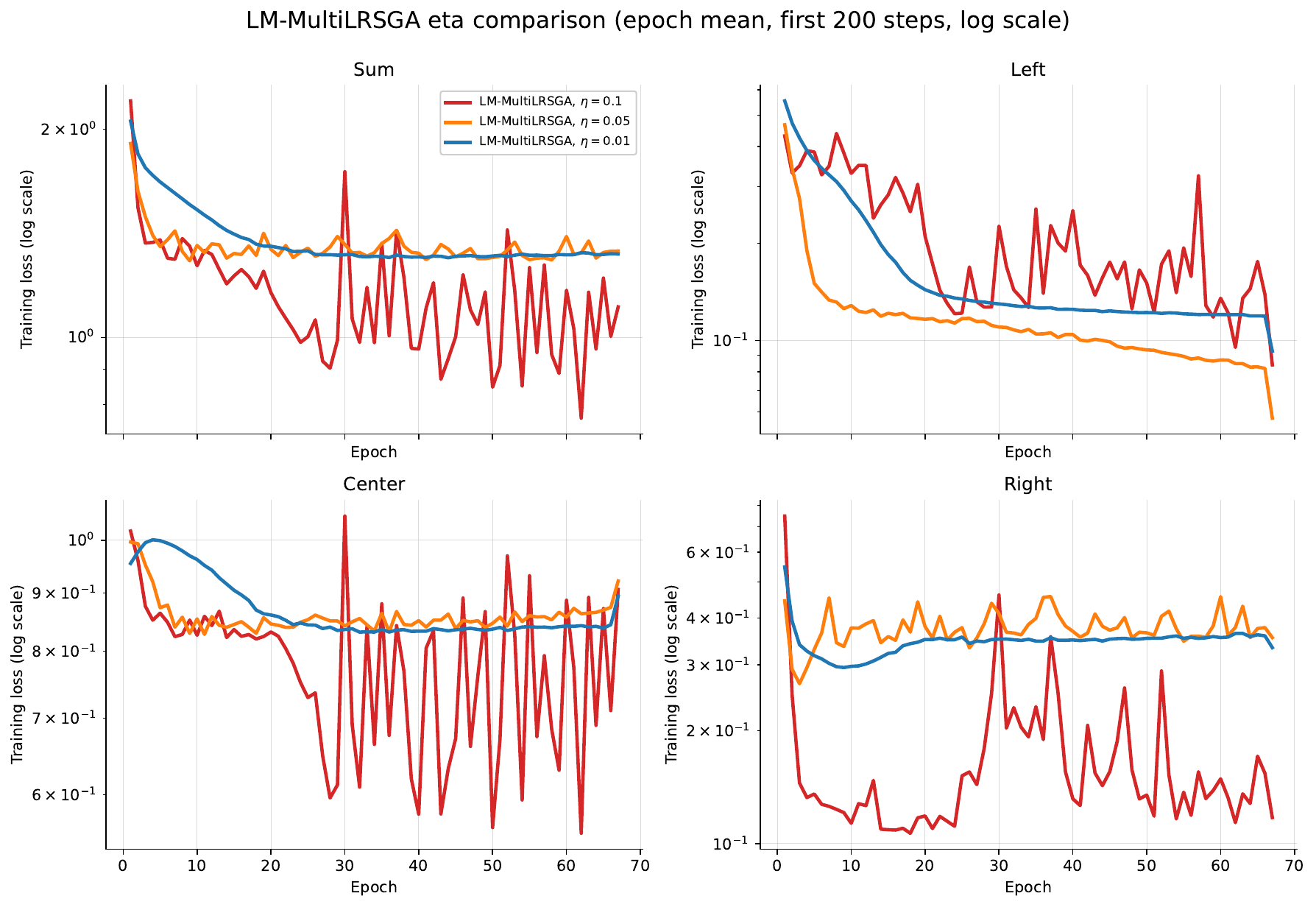}
    \caption{LM--MultiLRSGA-only step-size diagnostic, log scale, over the first \(200\)
iterations.}
    \label{fig:app_pinn_lm_eta_diagnostic_linear}
\end{figure}

\clearpage
\newpage
\subsection{LM--MultiLRSGA memory-history diagnostic}
\label{app:pinn_lm_history_diagnostic}

We also inspect the sensitivity of LM--MultiLRSGA to the memory-history
parameter. The diagnostic fixes
\begin{equation}
    \eta_{\mathrm{LM}}=0.1,
    \qquad
    \tau_{\mathrm{LM}}=0.01,
\end{equation}
and compares
\begin{equation}
    \texttt{history}\in\{3,5,10,20\}.
    \label{eq:app_lm_history_sweep}
\end{equation}
The plot version is visually dominated by occasional spikes. For this reason,
Table~\ref{tab:app_pinn_lm_history_epoch_mean} reports epoch-averaged
mini-batch losses at five checkpoints instead of using a figure. Each entry is
the mean \(\pm\) standard deviation over the five seeds used in the diagnostic.

\begin{table*}[hbt!]
\centering
\caption{
LM--MultiLRSGA memory-history diagnostic. Losses are epoch-averaged mini-batch
training losses at selected checkpoints. The five epoch checkpoints correspond
approximately to iterations \(90,180,270,360,\) and \(450\), since one
mini-batch epoch contains three iterations.
}
\label{tab:app_pinn_lm_history_epoch_mean}
\resizebox{\textwidth}{!}{%
\begin{tabular}{ccccccc}
\toprule
History & Epoch & Seeds & \(\Lleft\) & \(\Lcenter\) & \(\Lright\) & \(\Lsum\) \\
\midrule
\multicolumn{7}{c}{Epoch 30} \\
3  & 30 & 5 & $0.2338\pm0.1645$ & $0.8322\pm0.0473$ & $0.2183\pm0.1428$ & $1.2843\pm0.1322$ \\
5  & 30 & 5 & $0.2427\pm0.1592$ & $0.8414\pm0.0285$ & $0.2583\pm0.1551$ & $1.3424\pm0.0554$ \\
10 & 30 & 5 & $0.2398\pm0.1584$ & $0.8383\pm0.0352$ & $0.2601\pm0.1609$ & $1.3382\pm0.0630$ \\
20 & 30 & 5 & $0.2446\pm0.1568$ & $0.8466\pm0.0184$ & $0.2502\pm0.1466$ & $1.3413\pm0.0537$ \\
\midrule
\multicolumn{7}{c}{Epoch 60} \\
3  & 60 & 5 & $0.2874\pm0.1671$ & $0.8568\pm0.0973$ & $0.1475\pm0.1213$ & $1.2917\pm0.1981$ \\
5  & 60 & 5 & $0.2482\pm0.1532$ & $0.8214\pm0.1272$ & $0.1869\pm0.1386$ & $1.2565\pm0.2416$ \\
10 & 60 & 5 & $0.2730\pm0.1679$ & $0.8212\pm0.1260$ & $0.1451\pm0.1224$ & $1.2393\pm0.2340$ \\
20 & 60 & 5 & $0.2478\pm0.1889$ & $0.7974\pm0.1108$ & $0.1463\pm0.1214$ & $1.1914\pm0.2522$ \\
\midrule
\multicolumn{7}{c}{Epoch 90} \\
3  & 90 & 5 & $1.40\times10^{21}\pm3.12\times10^{21}$ & $7.63\times10^{26}\pm1.71\times10^{27}$ & $6.94\times10^{27}\pm1.55\times10^{28}$ & $7.71\times10^{27}\pm1.72\times10^{28}$ \\
5  & 90 & 5 & $0.3070\pm0.1847$ & $0.8852\pm0.0982$ & $0.3278\pm0.2845$ & $1.5200\pm0.2196$ \\
10 & 90 & 5 & $0.2381\pm0.2505$ & $0.7408\pm0.4548$ & $0.2098\pm0.3122$ & $1.1887\pm0.7725$ \\
20 & 90 & 5 & $0.2476\pm0.1785$ & $0.8973\pm0.1443$ & $0.2566\pm0.1925$ & $1.4015\pm0.3448$ \\
\midrule
\multicolumn{7}{c}{Epoch 120} \\
3  & 120 & 5 & $0.1066\pm0.1276$ & $0.5281\pm0.5017$ & $0.0568\pm0.0542$ & $0.6915\pm0.6655$ \\
5  & 120 & 5 & $0.2173\pm0.2224$ & $0.7101\pm0.4089$ & $0.1565\pm0.2301$ & $1.0839\pm0.6351$ \\
10 & 120 & 5 & $0.1867\pm0.1729$ & $0.6781\pm0.3965$ & $0.0882\pm0.0582$ & $0.9530\pm0.5745$ \\
20 & 120 & 5 & $0.1276\pm0.1356$ & $0.6430\pm0.3796$ & $0.2745\pm0.2952$ & $1.0451\pm0.6421$ \\
\midrule
\multicolumn{7}{c}{Epoch 150} \\
3  & 150 & 5 & $0.1466\pm0.1938$ & $0.5269\pm0.4882$ & $0.0427\pm0.0546$ & $0.7163\pm0.6703$ \\
5  & 150 & 5 & $0.2494\pm0.2653$ & $0.6773\pm0.6484$ & $0.0558\pm0.0680$ & $0.9825\pm0.9033$ \\
10 & 150 & 5 & $0.2567\pm0.2557$ & $0.8111\pm0.4772$ & $0.0941\pm0.0713$ & $1.1619\pm0.6740$ \\
20 & 150 & 5 & $0.1756\pm0.2316$ & $0.7860\pm0.5602$ & $0.3879\pm0.3593$ & $1.3495\pm0.8949$ \\
\bottomrule
\end{tabular}
}
\end{table*}

\clearpage
\newpage

\end{document}